\documentclass[11pt]{article}
\usepackage[dvips]{graphicx}
\usepackage{amsmath}
\usepackage{amsfonts}
\usepackage{chngpage}
\usepackage{titlesec}
\usepackage{fancyhdr}
\usepackage[dvips]{color}
\usepackage{multirow}
\makeatletter
\def\EquationsBySection{\def\theequation
{\thesection.\arabic{equation}}%
\@addtoreset{equation}{section}}

\makeatother \EquationsBySection
\allowdisplaybreaks

\newtheorem{theorem}{\bf Theorem}[section]
\newtheorem{proposition}[theorem]{\bf Proposition}

\newtheorem{corollary}[theorem]{\bf Corollary}

\newtheorem{lemma}[theorem]{\bf Lemma}
\newtheorem{algorithm}[theorem]{\bf Algorithm}
\newenvironment{proof}{Proof:}{\quad \hfill $\Box$ \vspace{2ex}}

\newtheorem{example}[theorem]{\bf Example}

\newtheorem{assumption}[theorem]{Assumption}

\begin{document}

\renewcommand\thepage{{\arabic{page}}}

{\renewcommand\thesection
    {{\thesection}}}

\changetext{}{}{-1.3cm}{-1.3cm}{}

\newcommand{\rme}{\mathrm{e}}
\newcommand{\rmi}{\mathrm{i}}
\newcommand{\rmd}{\mathrm{d}}
\newcommand{\norm}[1]{\left\Vert#1\right\Vert}
\newcommand{\abs}[1]{\left\vert#1\right\vert}
\newcommand{\set}[1]{\left\{#1\right\}}
\newcommand{\inner}[1]{\left\langle#1\right\rangle}
\newcommand{\Real}{\mathbb R}
\newcommand{\eps}{\varepsilon}
\newcommand{\To}{\rightarrow}
\newcommand{\BX}{\mathbf{B}(X)}
\newcommand{\A}{\mathcal{A}}
\renewcommand{\vec}[1]{\boldsymbol{#1}}
%


\title{Computing Highly Oscillatory Integrals\thanks{This research is supported in part by the US National Science Foundation under grant DMS-1115523, by Guangdong Provincial Government of China through the Computational Science Innovative Research Team program and
by the Natural Science Foundation of China under grants 11071286 and 91130009.}}
\author{
Yunyun Ma\thanks{Guangdong
Province Key Lab of Computational Science, School of Mathematics and Computational Science, Sun
Yat-sen University, Guangzhou 510275, P. R. China.
{\it mayy007@foxmail.com}, {\it xuyuesh@mail.sysu.edu.cn}. } \ and
Yuesheng Xu$^*$\thanks{Department of Mathematics, Syracuse University, Syracuse, NY
13244, USA.
 {\it yxu06@syr.edu}. All correspondence sent to this author.}  }
 \date{August 15,~2014}

\maketitle

\begin{abstract}
We develop two classes of composite moment-free numerical quadratures for computing highly oscillatory integrals having integrable singularities and stationary points. The first class of the quadrature rules has a polynomial order of convergence and the second class has an exponential order of convergence.
We first modify the moment-free Filon-type method for the oscillatory integrals without a singularity or a stationary point to accelerate their convergence. We then consider the oscillatory integrals without a singularity or a stationary point and then those with singularities and stationary points. The composite quadrature rules are developed based on partitioning the integration domain according to the wave number and the singularity of the integrand. The integral defined on a subinterval has either
a weak singularity without rapid oscillation or oscillation without a singularity.
The classical quadrature rules for weakly singular integrals using graded points are employed for the singular integral without rapid oscillation and the modified moment-free Filon-type method is used for the oscillatory integrals without a singularity. Unlike the existing methods, the proposed methods do not have to compute the inverse of the oscillator. Numerical experiments are presented to demonstrate the approximation accuracy and the computational efficiency of the proposed methods. Numerical results show that the proposed methods outperform methods published most recently.
\end{abstract}


\section{Introduction}

We consider in this paper numerical computation of highly oscillatory integrals defined on a bounded interval whose integrands have the form $f{\rm e}^{{\rm i}\kappa g}$, where the  wave number $\kappa$ is large, the amplitude function $f$ may have weak singularities, and the oscillator $g$ has stationary points of certain order. Computing highly oscillatory integrals is of importance in wide application areas ranging from quantum chemistry, computerized tomography, electrodynamics and fluid mechanics. For a large wave number $\kappa$, the integrands oscillate rapidly and cancel themselves over most of the range. Cancelation dose not occur in the neighborhoods of critical points of the integrand (the endpoints of the integration domain and the stationary points of the oscillator). Efficiency of a quadrature of highly oscillatory integrals depends on the behavior of functions $f$ and $g$ near the critical points. Traditional methods for evaluating oscillatory integrals become expensive when the wave number $\kappa$ is large, since the number of the evaluations of the integrand used grows linearly with the wave number $\kappa$ in order to obtain certain order of accuracy.
The calculation of the integrals is widely perceived as a {\it challenge} issue.
Calculating oscillatory integrals requires special effort.

The interest in the highly oscillatory integrals has led to much progress in developing numerical quadrature formulas for computing these integrals. In the literature, there are mainly four classes of methods for the computation: asymptotic methods, Filon-type methods, Levin-type methods and numerical steepest descent methods. The basis for convergence analysis of these quadrature rules is the asymptotic expansion of the oscillatory integral, an asymptotic expansion in negative powers of the wave number $\kappa$.
The leading terms in the asymptotic expansion may be derived from integration by parts \cite{Iserles} for the case when the oscillator has no stationary point. For the case when the oscillator has stationary points, the main tool is the method of stationary phase \cite{Olver, Stein}. For a fixed wave number, the convergence order of the asymptotic method is rather low. To overcome this weakness,
the Filon-type methods \cite{Filon, Flinn, Luke} were proposed, which replace the amplitude function $f$  by a suitable interpolating function. In many situations the convergence order of the Filon-type methods is significantly higher than that of the asymptotic methods. A thorough qualitative understanding of these methods and the analysis of their convergence order may be found in \cite{Iserles3, Iserles4, Iserles}
for the univariate case and in \cite{Iserles1} for the multivariate case.
In these methods, interpolation at the Chebyshev points ensures convergence \cite{Melenk}.
A drawback of the Filon-type methods is that they require to compute the moments, which themselves are oscillatory integrals. For the cases having nonlinear oscillators, it is not always possible to compute the moments exactly. In \cite{Graham1, Graham, Olver1, Xiang}, the moment-free Filon-type methods were developed.
An entirely different approach without computing the moments is the Levin collocation method \cite{Levin}. The Levin-type methods \cite{Levin1, Olver3, Olver5, Olver4} reduce computation of the oscillatory integral to a simple problem of finding the antiderivative $F$ of the integrand, where $F$ satisfies the differential equation $F'+{\rm i}\kappa g'F=f$. The Filon-type methods and the Levin-type methods with polynomials bases
are identical for the cases having the linear oscillator but not for the cases having the nonlinear oscillator \cite{Levin, Olver2, Xiang1}. Numerical steepest descent methods \cite{Asheim, Huybrechs, Huybrechs1} for removing the oscillation converts the real integration interval into a path in the complex plane, with a standard quadrature method used to calculate the resulting complex integral.

Although many methods were proposed for computing oscillatory integrals in the literatures,
there are still a big room for improving their approximation accuracy and computational efficiency.
The Filon-type and Levin-type methods require interpolating the derivatives of the amplitude $f$ at critical points in order to achieve a higher convergence order.
Even though computing derivatives can be avoided by allowing the interpolation points to
approach the critical points as the wave number increases for the formula proposed in \cite{Iserles2}, the moments cannot always  be explicitly computed. In particular, certain special functions were used for calculating the oscillatory integrals in the case when $f$ has singularities and $g$ has stationary points. The formulas proposed in \cite{Graham1, Graham} do not require computing the special functions and the moments of these formulas can be computed exactly, at the expenses of computing the inverse of the oscillator, which takes up much computing time. Numerical steepest descent methods also require computing the inverse of the oscillator or high order derivatives of the integrand \cite{Asheim}.

The purpose of this paper is to develop efficient composite quadrature rules for computing highly oscillatory integrals with singularities and stationary points. The methods to be described in this paper require neither computing the moments of the integrand, inverting the oscillator, nor calculating the derivatives of $f$ and those of $g'$.
The main idea used here is to divide the integration interval into subintervals according to the wave number $\kappa$ and the singularity of the integrand. To avoid using the special functions, we first split the integration interval into the subintervals according to the singularity of $f$ and the stationary points of $g$ such that the integrand on the subintervals either has a weak singularity but no oscillation, or has oscillation but no singularity or stationary point. The weakly singular integrals are calculated by the classical quadratures using graded points \cite{Xu2, Xu1, Schwab}.
To avoid using
the derivatives of $f$ and those of $g'$
and avoid computing the inverse of the oscillator,
we design a composite quadrature formulas using a partition of the subinterval formed according to the wave number $\kappa$ and the property of the oscillator $g$ for the oscillatory integrals. These formulas can improve the approximation accuracy effectively, since the convergence order of the formulas computing the oscillatory integrals with smooth integrand and without stationary point may be increased by adding more internal interpolation nodes. Specifically, we develop two classes of composite moment-free quadrature formulas for the highly oscillatory integrals. Class one uses a fixed number of quadrature nodes in each subinterval and has a polynomial order of convergence. This class of formulas are stable and easy to implement. Class two uses variate numbers of quadrature nodes in the subintervals and achieves an exponential order of convergence. Convergence order of this class of formulas is higher than that of the first class.

The quadrature formulas proposed in this paper have the following advantages. Comparing with the existing formulas, the proposed formulas need not computing the inverse of the nonlinear oscillator, or utilizing the incomplete Gamma function \cite{Abramowitz} for the oscillator with stationary points. These formulas not only  reduce the computational complexity, but also enhance the approximation accuracy. The approximation accuracy of these formulas is higher than that of the existing formulas for the case when the oscillator integral has stationary points and the oscillator is not easy to invert.

We organize this paper in seven sections. In Section 2, we present an improved moment-free Filon-type method for the oscillatory integrals developed in \cite{Xiang}. In Section 3, we design a partition of the integration interval and propose composite moment-free Filon-type methods for the oscillatory integrals with smoothing integrand and without a stationary point. In sections 4 and 5, we develop the composite moment-free quadratures defined on a mesh according to the wave number $\kappa$ and the properties of the integrand for the oscillatory integrals with both singularities and stationary points. The formulas proposed in Section 4 have a polynomial order of convergence, and those in Section 5 have an exponential order of convergence. Numerical experiments are presented in Section 6 to confirm the theoretical estimates on the accuracy of the proposed formulas. Moreover, we compare the numerical performance of the proposed formulas with that of those recently proposed in \cite{Graham1, Graham}.
We summarize our conclusions in Section 7.

\section{The Filon-type Quadrature Method}

The goal of this paper is to develop quadrature methods for evaluating oscillatory integrals in the form
\begin{equation}\label{Sec2:OI0}
\mathcal{I}_\kappa[f,g]:=\int_I f(x){\rm e}^{{\rm i}\kappa g(x)}{\rm d} x,
\end{equation}
where $I:=[0,1]$,  $\kappa\gg 1$ is the wave number, $f\in L^{1}(I)$ has weak singularities and the oscillator $g\in C^\infty(I)$ has stationary points.
Our main idea to fulfil this may be described as follows. We first develop a basic quadrature formula for computing an oscillatory integral defined on a subinterval $[a,b]$ of $I$ whose integrand has no singularities or stationary points. We then design an appropriate partition $0=x_0<x_1<\ldots<x_{n-1}<x_n=1$
of $I$ according to the wave number $\kappa$, the singularities of $f$  and the stationary points of $g$ and employ the basic quadrature formula for each of the integrals defined on the subintervals $[x_j, x_{j+1}]$, for $j\in \mathbb{Z}_{n-1}:=\{0,1,\ldots, n-1\}$.

We recall the Filon-type quadrature method proposed in \cite{Xiang} for computing the integral
\begin{equation}\label{Sec2:OI}
\mathcal{I}_\kappa^{[a,b]}[f,g]:=\int_a^bf(x){\rm e}^{{\rm i}\kappa g(x)}{\rm d}x,
\end{equation}
where $[a,b]\subset I$, $f$ is continuous on $[a,b]$ and the oscillator $g$ is continuously differentiable on $[a,b]$ and has no stationary point in $[a,b]$.
By a change of variables $g(x)\to x$, the integral in \eqref{Sec2:OI} may be written as
\begin{equation}\label{Sec2:OI1}
\mathcal{I}_\kappa^{[a,b]}[f,g]=\int_{g(a)}^{g(b)}\Psi(x){\rm e}^{{\rm i}\kappa x}{\rm d}x,
\end{equation}
where
$\Psi(x):=\left((f/g')\circ g^{-1}\right)(x)$, for
$x\in[g(a),g(b)]$.
For a fixed positive integer $m$ we approximate $\Psi$ by its Lagrange interpolation polynomial of degree $m$.
Since $g$ is differentiable on $[a,b]$ and has no stationary point in the interval, $g$ must be strictly monotone on the interval.
Without loss of generality, we assume that $g$ is strictly increasing and
otherwise we consider the integral $\mathcal{I}_\kappa^{[a,b]}[f,-g]$. Letting
$m\in\mathbb{N}:=\set{1, 2, \ldots}$,
we choose $m+1$ interpolation nodes in $[a,b]$ such that
$a=t_0<t_1<\ldots<t_{m-1}<t_m=b$.
Let $p_m$ denote the Lagrange polynomial which interpolates $\Psi$ at the points $g(t_j)$, $j\in \mathbb{Z}_{m}$. Hence, its Newton form is given by
$p_m=\sum_{j\in\mathbb{Z}_m}a_jw_j$
with $w_j(x)=\prod_{l\in\mathbb{Z}_{j-1}}(x-g(t_l))$ for $x\in[g(a),g(b)]$ and $j\in\mathbb{Z}_m$,
where the coefficients $a_j$, $j\in\mathbb{Z}_m$, are the divided differences of $\Psi$, that is,
$a_j:=\Psi[g(t_0),g(t_1),\ldots,g(t_j)]$. When computing $a_j$, we are required to compute the values of $\Psi$ at $g(t_j)$. Noting that $\left(g^{-1}\circ g\right)(t_j)=t_j$ for $j\in\mathbb{Z}_m$, we need only to evaluate the functional evaluations of $f$ and $g'$ at the points $t_j$.
There is no need to calculate the inverse of $g$.
A Filon-type quadrature rule is then obtained by replacing $\Psi$ in \eqref{Sec2:OI1} with $p_m$. That is, we use
\begin{equation}\label{Sec2:Filon}
\mathcal{Q}^{[a,b]}_{\kappa,m}[f,g]:=\sum_{j\in\mathbb{Z}_m}a_j\mathcal{I}^{[g(a), g(b)]}_\kappa[w_j,\widetilde{g}],
\end{equation}
where $\widetilde{g}(x)=x$ for $x\in[g(a),g(b)]$, to approximate the integral \eqref{Sec2:OI1}.
In formula \eqref{Sec2:Filon}, the integrals $\mathcal{I}^{[g(a), g(b)]}_\kappa[w_j,\widetilde{g}]$ for $j\in \mathbb{Z}_{m}$ can be computed exactly and efficiently.
We postpone the description of computing these integrals and turn our attention to error analysis of formula \eqref{Sec2:Filon}. For a function $\phi\in C(\Omega)$, let
$\norm{\phi}_{\infty}:=\max\limits_{x\in\Omega}\set{\abs{\phi(x)}}$. Note that in this paper $\Omega$ will be either $[a,b]$ or $[g(a),g(b)]$. Let
$\mathcal{E}_{\kappa,m}^{[a,b]}[f,g]:=\abs{\mathcal{I}_\kappa^{[a,b]}[f,g]-\mathcal{Q}^{[a,b]}_{\kappa,m}[f,g]}$
and $\sigma:=\norm{g'}_\infty$.
According to \cite{Xiang}, we have the following error estimate
\begin{equation}\label{Sec2:FilonError2}
\mathcal{E}_{\kappa,m}^{[a,b]}[f,g]\leq \dfrac{3(m+1)}{m!\kappa^2}\norm{\Psi^{(m+1)}}_{\infty}\sigma^m(b-a)^m.
\end{equation}
Estimate \eqref{Sec2:FilonError2} demonstrates that the decay of the error of the Filon-type method \eqref{Sec2:Filon} is of order $\mathcal{O}(\kappa^{-2})$. The decay  of \eqref{Sec2:Filon}
is also $\mathcal{O}(\kappa^{-2})$ for the linear oscillator (see, \cite{Iserles3, Iserles1, Xiang}).
For a quadrature formula $F$ that approximates integral \eqref{Sec2:OI1}, we use $\mathcal{N}(F)$ to denote  the number of evaluations of the integrand $\Psi$
used in the formula.  According to \eqref{Sec2:Filon}, we have that
$\mathcal{N}\left(\mathcal{Q}^{[a,b]}_{\kappa,m}[f,g]\right)\leq m+1$.

From \eqref{Sec2:FilonError2}, we see that convergence of the quadrature rule $\mathcal{Q}^{[a,b]}_{\kappa,m}[f,g]$ is  affected by $\sigma$.
When $\sigma=1$, the error bound in \eqref{Sec2:FilonError2} is not affected by $\sigma$ at all.
When $\sigma<1$, the error bound in \eqref{Sec2:FilonError2} decreases exponentially with respect to
$\sigma$ as $m\to\infty$. When $\sigma>1$, the error bound in \eqref{Sec2:FilonError2} grows exponentially with respect to $\sigma$ as $m\to\infty$.
For the case $\sigma>1$ to accelerate convergence
we subdivide $[a,b]$ into $N\in\mathbb{N}$  equal subintervals and approximate $\Psi$  by its Lagrange interpolation polynomial $p_{m}$ of degree $m\in\mathbb{N}$ on each of the subintervals.
For $N\in \mathbb{N}$, we use $y_j:=a+(b-a)j/N$ for $j\in\mathbb{Z}_{N}$ to denote the partition of $[a,b]$ and $\mathcal{Q}_{\kappa,m}^{[y_{j-1},y_j]}[f,g]$ computed by \eqref{Sec2:Filon} to approximate $\mathcal{I}_{\kappa}^{[y_{j-1},y_j]}[f,g]$  for $j\in\mathbb{Z}_{N}^+$.
This leads to the quadrature formula  $\mathcal{Q}_{N,\kappa,m}^{[a,b]}[f,g]$ for computing  \eqref{Sec2:OI1}, defined by
\begin{equation}\label{alm:Sec2Filon}
\mathcal{Q}^{[a,b]}_{N,\kappa,m}[f,g]:=\sum_{j\in\mathbb{Z}_N^+}\mathcal{Q}_{\kappa,m}^{[y_{j-1},y_j]}[f,g].
\end{equation}
Note that $\mathcal{Q}_{1,\kappa,m}^{[a,b]}[f,g]=\mathcal{Q}_{\kappa,m}^{[a,b]}[f,g]$. This quadrature formula will be used  in the following sections for computing the oscillatory integrals with the integrand without a singularity or a stationary point.

In the following theorem, we analyze the error
$\mathcal{E}_{N,\kappa,m}^{[a,b]}[f,g]
:=\abs{\mathcal{I}_\kappa^{[a,b]}[f,g]-\mathcal{Q}^{[a,b]}_{N,\kappa,m}[f,g]}$.

\begin{theorem}\label{thm:Sec2FilonM}
If $\Psi\in C^{m+1}[g(a),g(b)]$, then for $N\in\mathbb{N}$
\begin{equation}\label{Sec2:FilonM1}
\mathcal{E}_{N,\kappa,m}^{[a,b]}[f,g]\leq \dfrac{3(m+1)}{m!\kappa^2N^{m-1}}\norm{\Psi^{(m+1)}}_{\infty}\sigma^{m}(b-a)^m,
\end{equation}
and $\mathcal{N}\left(\mathcal{Q}_{N,\kappa,m}^{[a,b]}[f,g]\right)\leq Nm+1.$
In particular, if $N$ is chosen as $\lceil\sigma\rceil$,  the smallest integer not less than $\sigma$, then
\begin{equation}\label{Sec2:FilonM2}
\mathcal{E}_{N,\kappa,m}^{[a,b]}[f,g]\leq \dfrac{3(m+1)}{m!\kappa^2}\norm{\Psi^{(m+1)}}_{\infty}\sigma(b-a)^m.
\end{equation}
\end{theorem}
\begin{proof}
Let $N\in\mathbb{N}$.
We estimate the error
$\mathcal{E}_{\kappa,m}^{[y_{j-1},y_j]}[f,g]:=
\abs{\mathcal{I}_{\kappa}^{[y_{j-1},y_j]}[f,g]-\mathcal{Q}_{\kappa,m}^{[y_{j-1},y_j]}[f,g]}$
by first employing estimate \eqref{Sec2:FilonError2} with $[a,b]$ replaced by $[y_{j-1},y_j]$ and then
Summing up both sides of the resulting inequality over $j\in\mathbb{N}_N^+$. Noting $y_j-y_{j-1}=(b-a)/N$, this leads to estimate \eqref{Sec2:FilonM1}.
According to the algorithm \eqref{alm:Sec2Filon}, noting that the nodes $y_j$ for $j\in\mathbb{Z}_{N-1}^+$ are used twice in the algorithm, we conclude that
$\mathcal{N}\left(\mathcal{Q}_{N,\kappa,m}^{[a,b]}[f,g]\right)\leq N\mathcal{N}\left(\mathcal{Q}_{1,\kappa,m}^{[a,b]}[f,g]\right)-(N-1)\leq Nm+1$.
In particular, when $N$ is chosen as $\lceil\sigma\rceil$, we substitute $N=\lceil\sigma\rceil$ into estimate \eqref{Sec2:FilonM1} to yield estimate \eqref{Sec2:FilonM2}.
\end{proof}

Comparing estimate \eqref{Sec2:FilonM1} in Theorem \ref{thm:Sec2FilonM} with estimate \eqref{Sec2:FilonError2}, we see that formula \eqref{alm:Sec2Filon} uses $Nm+1$ number of the functional evaluations of $\Psi$, which is $N$ times of that used in the traditional Filon method to reach the order of error estimate $N^{-m+1}\mathcal{E}_{\kappa,m}^{[a,b]}[f,g]$. Formula \eqref{alm:Sec2Filon} will serve as a basic quadrature formula in this paper for developing sophisticated formulas for computing singular oscillatory integrals.

In the remaining sections of this paper, we shall consider the following three cases:
\begin{enumerate}
  \item[(i)]{When $f$ and $g$ are smooth and $g$ has no stationary point or inflection point in $I$, according to the wave number $\kappa$ we design a partition $0=x_0<x_1<\ldots<x_n=1$, and write $\mathcal{I}_\kappa[f,g]=\sum_{j\in\mathbb{Z}_{n}^+}\mathcal{I}_\kappa^{[x_{j-1},x_{j}]}[f,g]$.
      Formula \eqref{alm:Sec2Filon} is then used to compute integrals  $\mathcal{I}_\kappa^{[x_{j-1},x_{j}]}[f,g]$ for $j\in\mathbb{Z}_{n}^+$.}

  \item[(ii)]{When $f$ has a weak singularity only at the origin and $g$ is smooth without a stationary point or an inflection point in $I$, we first divide $I$ into two subintervals $[0,b]$ and $\Lambda:=[b,1]$ such that the integrand of $\mathcal{I}_\kappa^{[0,b]}[f,g]$ does not rapidly oscillate and that of $\mathcal{I}_\kappa^{\Lambda}[f,g]$ does not have singularity. The integral  $\mathcal{I}_\kappa^{[0,b]}[f,g]$ is calculated by a quadrature rule for weak singular integrals.  The integral  $\mathcal{I}_\kappa^{\Lambda}[f,g]$ is computed by the method described in item (i).}

  \item[(iii)]{When $f$ has a weak singularity only at the origin and  $g$ is smooth with one stationary point at the origin and has no inflection point in $I$,
      we first divide $I$ into two subintervals $[0,b]$ and $\Lambda$ such that the integrand of $\mathcal{I}_\kappa^{[0,b]}[f,g]$ does not rapidly oscillate and that of $\mathcal{I}_\kappa^{\Lambda}[f,g]$  does not have singularity or stationary point. The integrals $\mathcal{I}_\kappa^{[0,b]}[f,g]$ and $\mathcal{I}_\kappa^{\Lambda}[f,g]$ are handled in the same way as (ii).}
\end{enumerate}

The case of $f$ having a finite number of singularities in $I$, $g$ having a finite number of  stationary points or inflection points in $I$ can be treated by splitting $I$ into subintervals, on each of which $f$ has only one singular point or $g$ has only one stationary point at an end-point and without an inflection point.

\section{The Composite Filon-type Quadrature Method}

In this section, we develop composite Filon-type qadrature methods for computing the oscillatory integrals  \eqref{Sec2:OI0}. Specifically, we assume that $f$, $g\in C^2(I)$ and $g$ is increasing  monotonically  with $\norm{g'}_\infty=\sigma$, $g'(x)\neq0$ for $x\in I$ and $g''(x)\neq 0$ for $x\in (0,1)$.
We shall partition the interval $I$ into $n$ subintervals according to the wave number $\kappa$,
and propose a composite quadrature rule, where we approximate $\Psi$ by
its Lagrange interpolation polynomial of variable degrees in each subinterval,
aiming at the asymptotic error order $\mathcal{O}(\kappa^{-n-1})$.

We first motivate the construction of a $\kappa$-dependent partition of $I$. By a change of variables of $\kappa x\to x$, the integral \eqref{Sec2:OI0} becomes
\begin{equation}\label{Sec2:P}
\mathcal{I}_\kappa[f,g]
=1/\kappa\int_0^1f(x/\kappa){\rm e}^{{\rm i}\kappa g(x/\kappa)}{\rm d}x+1/\kappa\int_1^\kappa f(x/\kappa){\rm e}^{{\rm i}\kappa g(x/\kappa)}{\rm d}x.
\end{equation}
The integrals on the right hand side of \eqref{Sec2:P} do not oscillate rapidly since
$\abs{\left(\kappa g(x/\kappa)\right)'}\leq \sigma$ for $x\in [0, \kappa]$. However, traditional quadratures for
computing the second integral on the right hand side of \eqref{Sec2:P} lead to prohibitive costs for a large $\kappa$. Inspired by the  quadratures for singular integrals using graded points proposed in \cite{Xu1},
for $n\in\mathbb{N}$ with $n>1$ we suggest the partition of $[1,\kappa]$ with graded points $\kappa^{(j-1)/(n-1)}$, for $j\in\mathbb{Z}_n^+$. Using the change of variables $x/\kappa\to x$, we obtain a partition of $I$.

We now describe the construction of the $\kappa$-dependent partition of $I$. To this end,
we fix $\kappa$. For $n\in\mathbb{N}$ with $n>1$, let $\Pi_\kappa$ denote
the partition of $I$ with nodes defined  by
\begin{equation}\label{Sec3:Partition}
  x_0=0,~x_j=\kappa^{(j-1)/(n-1)-1},~\text{for}~ j\in\mathbb{Z}_{n}^+.
\end{equation}
According to the partition $\Pi_\kappa$, the integral \eqref{Sec2:OI0} is written as
$\mathcal{I}_{\kappa}[f,g]=\sum_{j\in\mathbb{Z}_{n}^+}\mathcal{I}_{\kappa}^{[x_{j-1},x_{j}]}[f,g]$.
Computing the integral $\mathcal{I}_{\kappa}[f,g]$ is then reduced to computing the integrals $\mathcal{I}_{\kappa}^{[x_{j-1},x_{j}]}[f,g]$ for $j\in\mathbb{Z}_{n}^+$. We shall use formula \eqref{alm:Sec2Filon} to calculate these integrals. For this purpose, we define the quantities
\begin{equation*}
M_j:=\max\set{\abs{g'(x_{j-1})},\abs{g'(x_j)}}~\text{and}~  N_j:=\lceil M_j\rceil,~\text{for}~j\in\mathbb{Z}_n^+.
\end{equation*}
Since $g$ is monotonically increasing on $I$, we have for $j\in\mathbb{Z}_n^+$ that
\begin{equation}\label{Sec3:der}
M_j=\max\set{\abs{g'(x)}: x\in[x_{j-1},x_{j}]}~\text{and}~ M_j\leq\sigma.
\end{equation}
We shall develop two quadrature methods. Method one uses a fixed number of quadrature points in each of the subintervals and has a polynomial order (in terms of the wave number) of convergence. Method two uses variable number of quadrature points in the subintervals and achieves an exponential order (in terms of the wave number) of convergence.

We first describe the method having a polynomial convergence order. We choose a fixed positive integer $m$.
For each $j\in\mathbb{Z}^{+}_{n}$, we use $\mathcal{Q}_{N_j,\kappa,m}^{[x_{j-1},x_{j}]}[f,g]$ defined as in \eqref{alm:Sec2Filon} to approximate $\mathcal{I}_\kappa^{[x_{j-1},x_{j}]}[f,g]$. Integral $\mathcal{I}_\kappa[f,g]$ defined by \eqref{Sec2:OI0} is then approximated by the quadrature formula
\begin{equation}\label{alm:Sec3P}
\mathcal{Q}_{\kappa,n,m}[f,g]:=\sum_{j\in\mathbb{Z}_n^+}\mathcal{Q}_{N_j,\kappa,m}^{[x_{j-1},x_{j}]}[f,g].
\end{equation}
In the next proposition, we provide an estimate of the error $\mathcal{E}_{\kappa,n,m}[f,g]:=\abs{\mathcal{I}_\kappa[f,g]-\mathcal{Q}_{\kappa,n,m}[f,g]}$
and the number $\mathcal{N}\left(\mathcal{Q}_{\kappa,n,m}[f,g]\right)$ of functional evaluations used in the formula $\mathcal{Q}_{\kappa,n,m}[f,g]$.

\begin{proposition}\label{prop:Sec3P}
For $n\in\mathbb{N}$ with $n>1$, let $\eta:=\max\set{1/\kappa,1-\kappa^{-1/(n-1)}}$. If $\Psi\in C^{m+1}[g(0),g(1)]$ for some $m\in\mathbb{N}$, then
\begin{equation*}\label{Sec3:Pk}
\mathcal{E}_{\kappa,n,m}[f,g]\leq \dfrac{3(m+1)}{m!\kappa^{2}}\norm{\Psi^{(m+1)}}_{\infty}\sigma\eta^{m-1},
\end{equation*}
and  $\mathcal{N}\left(\mathcal{Q}_{\kappa,n,m}[f,g]\right)\leq\left\lceil\sigma\right\rceil nm+1$.
\end{proposition}
\begin{proof}
The proof is done by applying Theorem \ref{thm:Sec2FilonM} on each of the subintervals $[x_{j-1},x_{j}]$.
Specifically, for $j\in\mathbb{Z}_n^+$, we use Theorem \ref{thm:Sec2FilonM} to estimate
$\mathcal{E}_{j}[f,g]:=
\abs{\mathcal{I}_{\kappa}^{[x_{j-1},x_{j}]}[f,g]-\mathcal{Q}_{N_j,\kappa,m}^{[x_{j-1},x_{j}]}[f,g]}$.
For $j\in\mathbb{Z}_{n}^+$, we apply \eqref{Sec2:FilonM2} with $\sigma$ being replaced by $M_j$ and $b-a$ by $h_j:=x_j-x_{j-1}$ to conclude that
$\mathcal{E}_{j}[f,g]\leq\dfrac{3(m+1)}{m!\kappa^2}\norm{\Psi^{(m+1)}}_{\infty}M_jh_j^{m}$.
Note that
$h_1=x_1-x_0=1/\kappa\leq\eta$, and for $j\in\mathbb{Z}_n^+$ with $j>1$,
$h_j=x_j-x_{j-1}=x_j\left(1-\kappa^{-1/(n-1)}\right)\leq x_j\eta\leq \eta$, and
$M_j\leq \sigma$. Substitute these bounds into the inequality above and summing the resulting inequality over $j\in\mathbb{Z}_n^+$, we obtain the desired estimate for $\mathcal{E}_{\kappa,n,m}[f,g]$.

Using Theorem  \ref{thm:Sec2FilonM} again yields that
$\mathcal{N}\left(\mathcal{Q}_{\kappa,n,m}[f,g]\right)\leq\sum_{j\in\mathbb{Z}_n^+}(N_jm+1)-(n-1)
\leq \left\lceil\sigma\right\rceil nm+1$.
\end{proof}

As a direct consequence of Proposition \ref{prop:Sec3P}, we have the next estimates for the case having the linear oscillator $g(x)=x$, $x\in I$, where $\sigma=1$.

\begin{corollary}\label{cor:Sec3LinearP}
For $n\in\mathbb{N}$ with $n>1$, let $\eta:=\max\set{1/\kappa,1-\kappa^{-1/(n-1)}}$.
If $f\in C^{m+1}(I)$ for some $m\in\mathbb{N}$ and $g(x)=x$ for $x\in I$, then
$\mathcal{E}_{\kappa,n,m}[f,g]\leq \dfrac{3(m+1)}{m!\kappa^{2}}\norm{f^{(m+1)}}_{\infty}\eta^{m-1}$
and
$\mathcal{N}\left(\mathcal{Q}_{\kappa,n,m}[f,g]\right)\leq  nm+1$.
\end{corollary}

We now turn to developing the quadrature formula having an exponential convergence order.
This is done by choosing variable numbers of quadrature nodes in the subintervals of $I$. Specifically,
for  $n\in\mathbb{N}$ with $n>1$ and for the partition of $I$ chosen as \eqref{Sec3:Partition}, we let
$m_j:=\left\lceil n(n-1)/(n+1-j)\right\rceil$ for $j\in\mathbb{Z}_{n}^+$.
For each $j\in\mathbb{Z}^{+}_{n}$, we use $\mathcal{Q}_{N_j,\kappa,m_j}^{[x_{j-1},x_{j}]}[f,g]$ to approximate $\mathcal{I}_\kappa^{[x_{j-1},x_{j}]}[f,g]$.
Integral $\mathcal{I}_\kappa[f,g]$ defined by \eqref{Sec2:OI0}
is then approximated by the quadrature formula
\begin{equation}\label{alm:Sec3E}
\mathcal{Q}_{\kappa,n}[f,g]:=\sum_{j\in\mathbb{Z}_n^+}\mathcal{Q}^{[x_{j-1},x_j]}_{N_j,\kappa,m_j}[f,g].
\end{equation}

We next study the error
$\mathcal{E}_{\kappa,n}[f,g]:=\abs{\mathcal{I}_\kappa[f,g]-\mathcal{Q}_{\kappa,n}[f,g]}$
of the quadrature $\mathcal{Q}_{\kappa,n}[f,g]$. To this end, we first establish two technical lemmas.

\begin{lemma}\label{lem:Sec3OrderP}
There exists a positive constant $c$  such that for all $n\in\mathbb{N}$ with $n>2$,
$\sum_{j\in\mathbb{Z}_{n}^+}1/(m_j-1)\leq c$.
\end{lemma}
\begin{proof}
We prove this result by estimating the  lower bound of the set
$\set{m_j-1: j\in\mathbb{Z}_{n}^+}$. For $j\in\mathbb{Z}_{n}^+$, we have that
$m_j-1\geq n(n-1)/(n+1-j)-1\geq n-2$.
Thus, we obtain that
$\sum_{j\in\mathbb{Z}_{n}^+}1/(m_j-1)\leq n/(n-2)$, which is bounded by a constant $c$.
\end{proof}

\begin{lemma}\label{lem:Sec3Pre}
There exists a positive constant $c$ such that for all $\kappa>1$, $n\in\mathbb{N}$ that satisfy
\begin{equation}\label{Sec3:C}
(n-1)(\ln{(n-2+{\rm e})}-1)\geq\ln{\kappa},
\end{equation}
and  $j\in\mathbb{Z}_n^+$ with $j>1$,
\begin{equation*}
(\kappa^{1/(n-1)}-1)^{m_j-2}/(m_j-2)!\leq c(n-2)^{-1/2}.
\end{equation*}
\end{lemma}
\begin{proof}
By the definition of $m_j$, we see that  $m_j-2\geq n-2$ for $j>1$. Condition \eqref{Sec3:C} implies that  $n>2$ since $\kappa>1$.
We observe from the Stirling formula \cite{Abramowitz} for $n\in\mathbb{N}$
\begin{equation}\label{Sec3:Stirling-coro}
n!\geq\sqrt{2\pi n}\left(n/{\rm e}\right)^n.
\end{equation}
Using inequality \eqref{Sec3:Stirling-coro} with $n:=m_j-2$, we conclude that there exists a positive constant $c$ such that for all $n\in\mathbb{N}$ with $n>2$ and $j\in\mathbb{Z}_n^+$ with $j>1$,
$1/(m_j-2)!\leq c (m_j-2)^{-1/2}\left({\rm e}/(m_j-2)\right)^{m_j-2}\leq c (n-2)^{-1/2}\left({\rm e}/(n-2)\right)^{m_j-2}$.
On the other hand, condition \eqref{Sec3:C} implies that
${\rm e}(\kappa^{1/(n-1)}-1)\leq n-2$.
This together with the above inequality ensures that there exists a positive
constant $c$ such that for all $\kappa>1$,
$n\in\mathbb{N}$ satisfying \eqref{Sec3:C} and  $j\in\mathbb{Z}_n^+$ with $j>1$,
$(\kappa^{1/(n-1)}-1)^{m_j-2}/(m_j-2)!
\leq c(n-2)^{-1/2}\left({\rm e}(\kappa^{1/(n-1)}-1)/(n-2)\right)^{m_j-2}\leq c(n-2)^{-1/2}$.
 This concludes the desired result.
\end{proof}

For a function $\phi\in C^{\infty}(\Omega)$, we let $\norm{\phi}_{n}:=\max\set{\norm{\phi^{(j)}}_\infty: j\in\mathbb{Z}_n}$ for $n\in\mathbb{N}$.
We are now ready to establish the estimate for $\mathcal{E}_{\kappa,n}[f,g]$.

\begin{theorem}\label{thm:Sec3E}
If $\Psi\in C^{\infty}[g(0),g(1)]$, then there exists a positive constant $c$ such that for all $\kappa> 1$ and  $n\in\mathbb{N}$ satisfying \eqref{Sec3:C},
\begin{equation*}\label{Sec3:E}
\mathcal{E}_{\kappa,n}[f,g]\leq c\sigma (n-2)^{-1/2}\kappa^{-n-1}\norm{\Psi}_{(n-1)n+1}.
\end{equation*}
For $n\in\mathbb{N}$ with $n>2$, there holds the estimate
$\mathcal{N}\left(\mathcal{Q}_{\kappa,n}[f,g]\right)\leq\left\lceil\sigma\right\rceil \left(n(n-1)\ln{n}+n^2\right)+1$.
\end{theorem}
\begin{proof}
We establish the error bound by estimating errors
$\mathcal{E}_{j}[f,g]:=
\abs{\mathcal{I}_{\kappa}^{[x_{j-1},x_{j}]}[f,g]-\mathcal{Q}_{N_j,\kappa,m_j}^{[x_{j-1},x_{j}]}[f,g]}$
for $j\in\mathbb{Z}_{n}^+$,
and then sum them over $j$.
By employing \eqref{Sec2:FilonM2} with $b-a$ being replaced by $h_j$ and $m$ by $m_j$ and \eqref{Sec3:der}, we obtain for $j\in\mathbb{Z}_{n}^+$ that
\begin{equation}\label{Sec3:proof1thmE}
\mathcal{E}_j[f,g]\leq\dfrac{3\sigma(m_j+1)}{m_j!\kappa^2} h_j^{m_j}\norm{\Psi^{(m_j+1)}}_{\infty}.
\end{equation}
For $j=1$, we have that
$\mathcal{E}_1[f,g]\leq\dfrac{6\sigma \kappa^{-n-1}}{(n-2)!}\norm{\Psi^{(n)}}_{\infty}$.
For $j>1$, by substituting $h_j=\kappa^{\frac{j-n-1}{n-1}}(\kappa^{1/(n-1)}-1)$  into \eqref{Sec3:proof1thmE} we obtain that
\begin{eqnarray}\label{Sec3:proof3thmE}
\mathcal{E}_j[f,g]\leq\dfrac{6\sigma}{m_j-1}
\dfrac{(\kappa^{1/(n-1)}-1)^{m_j-2}}{(m_j-2)!}\left(\kappa^{1/(n-1)}-1\right)^{2}\kappa^{\frac{j-n-1}{n-1}m_j-2}
\norm{\Psi^{(m_j+1)}}_{\infty}.
\end{eqnarray}
On the other hand, condition \eqref{Sec3:C} implies that $n>2$. Thus, $\left(\kappa^{1/(n-1)}-1\right)^{2}<\kappa$.
Applying Lemma \ref{lem:Sec3Pre} to \eqref{Sec3:proof3thmE} yields
a positive constant $c$ such that for all $\kappa>1$, $n\in\mathbb{N}$ satisfying \eqref{Sec3:C}, and  $j\in\mathbb{Z}_{n}^+$ with $j>1$,
$\mathcal{E}_j[f,g]\leq c\sigma\dfrac{(n-2)^{-1/2}}{m_j-1}\kappa^{\frac{j-n-1}{n-1}m_j-1}\norm{\Psi^{(m_j+1)}}_{\infty}$.
Since for $j\in\mathbb{Z}_{n}^+$ with $j>1$,
$\dfrac{j-n-1}{n-1}m_j-1\leq \dfrac{j-n-1}{n-1}\dfrac{n(n-1)}{n+1-j}-1=-n-1$,
we conclude
that there exists a positive constant $c$ such that for all $\kappa>1$, $n\in\mathbb{N}$ satisfying \eqref{Sec3:C}, and $j\in\mathbb{Z}_{n}^+$ with $j>1$,
$\mathcal{E}_j[f,g]\leq c\sigma\dfrac{(n-2)^{-1/2}}{m_j-1}\kappa^{-n-1}\norm{\Psi^{(m_j+1)}}_{\infty}$.
Summing up the bound of errors $\mathcal{E}_j[f,g]$ over $j\in\mathbb{Z}_{n}^+$ and applying Lemma \ref{lem:Sec3OrderP},  we observe that
there exists a positive constant $c$ such that for all $\kappa> 1$ and $n\in\mathbb{N}$ satisfying \eqref{Sec3:C}
\begin{eqnarray*}
\mathcal{E}_{\kappa,n}[f,g]
&\leq&\dfrac{6\sigma}{(n-2)!}\kappa^{-n-1}\norm{\Psi^{(n)}}_{\infty}+c\sigma\kappa^{-n-1}
       \set{\sum_{j=2}^{n}\dfrac{(n-2)^{-{1}/{2}}}{m_j-1}}\norm{\Psi}_{(n-1)n+1}\\
&\leq&c\sigma(n-2)^{-1/2}\kappa^{-n-1}\norm{\Psi}_{(n-1)n+1}.
\end{eqnarray*}

It remains to estimate the number of functional evaluations used in the quadrature formula. To this end, we note that
 Theorem  \ref{thm:Sec2FilonM} yields
\begin{eqnarray*}
\mathcal{N}\left(\mathcal{Q}_{\kappa,n}[f,g]\right)
&\leq&\sum_{j\in\mathbb{Z}_n^+}\Big{\{}\left\lceil\sigma\right\rceil
      \left(n(n-1)/(n+1-j)+1\right)+1\Big{\}}-(n-1)\\
&\leq&\left\lceil\sigma\right\rceil\left\{n+n(n-1)
      \sum_{j\in\mathbb{Z}_n^+}1/(n+1-j)\right\}+1.
\end{eqnarray*}
For $n\in\mathbb{N}$ by using $\sum_{j\in\mathbb{Z}_n^+}1/j\leq \ln{n}+1$ we have that
\begin{eqnarray*}
\mathcal{N}\left(\mathcal{Q}_{\kappa,n}[f,g]\right)  \leq\left\lceil\sigma\right\rceil\left\{n+n(n-1)(\ln(n)+1)\right\}+1
  = \left\lceil\sigma\right\rceil \left(n(n-1)\ln{n}+n^2\right)+1,
\end{eqnarray*}
which completes the proof.
\end{proof}

As a direct consequence of Theorem \ref{thm:Sec3E}, we have the next estimates for the case having the linear oscillator.

\begin{corollary}
If $f\in C^{\infty}(I)$, then there exists a positive constant $c$  such that for all $\kappa> 1$ and $n\in\mathbb{N}$ satisfying \eqref{Sec3:C},
$\mathcal{E}_{\kappa,n}[f,g]\leq c (n-2)^{-{1}/{2}}\kappa^{-n-1}\norm{f}_{(n-1)n+1}$.
For $n\in\mathbb{N}$ with $n>2$,  there holds the estimate
$\mathcal{N}\left(\mathcal{Q}_{\kappa,n}[f,g]\right)\leq n(n-1)\ln{n}+n^2+1$.
\end{corollary}

The Filon-type methods \cite{Iserles, Xiang} and the Levin-type methods \cite{Levin, Olver4} achieving a convergence order higher than $\mathcal{O}(\kappa^{-2})$ require the evaluation of derivatives of $f$ and $g'$. The Filon-Clenshaw-Curtis rules \cite{Graham} for the oscillatory integrals with nonlinear oscillator requires the evaluation of $g^{-1}$. Quadrature methods developed in this section do not require computing derivatives of $f$ or those of $g'$, nor evaluating  ${g}^{-1}$. Moreover,
since condition \eqref{Sec3:C} implies that $\ln{\kappa}\leq n\ln{n}$, Theorem \ref{thm:Sec3E} demonstrates  that the quadrature $\mathcal{Q}_{\kappa,n}[f,g]$ achieves the asymptotic convergence order $\mathcal{O}(n^{-1/2}\kappa^{-n-1})$ with only $\mathcal{O}(\ln^2{\kappa})$ number of functional evaluations.

\section{Quadratures with a Polynomial Order of Convergence}

In this section, we consider computing the oscillatory integral \eqref{Sec2:OI0},
where $f$ is allowed to have weak singularities, $g$  has  stationary points and has  no
inflection point in $I$. We shall develop quadrature formulas having a polynomial order of convergence.

The key idea to be employed is to split the interval $I$ into two subintervals such that
on one subinterval the integrand has a weak singularity but no oscillation and
on the other subinterval it has oscillation but no singularity or stationary point. We then
employ quadratures using graded points for the singular integral and design a composite quadrature rules using a partition of the subinterval which is formed according to the wave number $\kappa$ and the property of $g$ for the oscillatory integral.

We first describe the weak singularity of a function defined on $I$ according to \cite{Xu1}. Let $S$ be a subset of $I$ containing a finite number of points. For some $\alpha\in (-1,1)$ and a nonnegative
integer $m$, a real-valued function  $f\in C^{m}(I\backslash S)$ is said to be of ${\rm Type}(\alpha,m,S)$ if
there exists a positive constant $c$ such that for all $x\in I\backslash S$,
\begin{equation}\label{Sec4:singularityDef}
  \abs{f^{(m)}(x)}\leq c\left[\omega_{S}(x)\right]^{\alpha-m},
\end{equation}
where  the function $\omega_S$ associated with $S$ is defined by $\omega_S(x)=\inf{\set{\abs{x-t}: ~t\in S}}$ for $x\in I$.
When we say that $f$ is of ${\rm Type}(\alpha,\infty,S)$, we mean that for  all $m\in\mathbb{N}$, $f\in C^{\infty}(I\backslash S)$  satisfies \eqref{Sec4:singularityDef}.
The parameter $\alpha$ is called the index of singularity. For $\alpha>0$, this notation was introduced by Rice in \cite{Rice}.

In this section,  $f$ is allowed to have a single weakly singular point at zero with index $\mu\in(-1,1)$ and  $g$ satisfies the following assumption:
\begin{assumption}\label{Sec4:Assume}
For a nonnegative integer $r$, the  function $g\in C^{r+1}(I)$ has a single stationary point at zero  satisfying
$g^{(j)}(0)=0$ for $j\in \mathbb{Z}_r$, $g^{(r+1)}(x)\neq0$ for $x\in I$,
and $\sigma(r):={\norm{g^{(r+1)}}_{\infty}}/{(r+1)!}\ll \kappa$, and $g$ is increasing monotonically without an inflection point in $I$.
\end{assumption}

The requirement that $g(0)=0$ in Assumption \ref{Sec4:Assume} is without loss of generality,
since if $g(0)\neq 0$,
we compute instead the integral
$\mathcal{I}_{\kappa}[f,g-g(0)]+\exp{\set{{\rm i}\kappa g(0)}}\mathcal{I}_{\kappa}[f,0]$.
For the case $r=0$, the oscillator $g$  does not have a stationary point.

We now write the integral \eqref{Sec2:OI0} as the sum of two integrals: a weakly singular integral without rapid oscillation and an oscillatory integral without a singularity or a stationary point.
According to the assumption on $g$, by the Taylor theorem, for each $x\in I$ there exists a constant $\xi_x\in[0,(\kappa\sigma(r))^{-1/(r+1)}x]$ such that
$\kappa\abs{g\left((\kappa\sigma(r))^{-1/(r+1)}x\right)}=\abs{{g^{(r+1)}(\xi_x)}x^{r+1}/\sigma(r)}/(r+1)!$.
Hence, we see that $\kappa |g(x)|\leq 1$ for
$0\leq x\leq(\kappa\sigma(r))^{-1/(r+1)}$.
Thus, for such an $x$ the function ${\rm e}^{{\rm i}\kappa g(x)}$ does not oscillate rapidly.
Motivated from the above discussion, we introduce
\begin{equation*}\label{Sec4:waveNumber}
\kappa_{\sigma(r)}:=
  \begin{cases}
    \kappa,          &\sigma(r)\leq1,\\
    \kappa\sigma(r), &\sigma(r)>1,
  \end{cases}
\end{equation*}
and define
\begin{equation}\label{Sec4:kappa}
\lambda_r:=\kappa_{\sigma(r)}^{-1/(r+1)}.
\end{equation}
We  split the interval $I$ into two subintervals $[0,\lambda_r]$ and $[\lambda_r,1]$.
Correspondingly, integral  \eqref{Sec2:OI0} may be written as the sum of integrals on these two subintervals.
Let $\Lambda:=[\lambda_r,1]$ and for $\phi\in L^1(I)$ we set
$\mathcal{I}[\phi]:=\int_0^1\phi(x){\rm d}x$.
Using a change of variables: $y=\lambda_r^{-1}x$ for
$x\in[0,\lambda_r]$,
the integral \eqref{Sec2:OI0} is rewritten as
$\mathcal{I}_{\kappa}[f,g]=\mathcal{I}^{\mu}[f,g]+\mathcal{I}_{\kappa}^{\Lambda}[f,g]$,
where
\begin{equation}\label{Sec4:weakSI}
\mathcal{I}^{\mu}[f,g]:=\lambda_r\mathcal{I}[\varphi_\kappa],
\end{equation}
with
\begin{equation}\label{Sec4:integrandSI}
\varphi_\kappa(x):=f\left(\lambda_rx\right){\rm e}^{{\rm i}\kappa
g(\lambda_rx)},~ \text{for}~x\in I
\end{equation}
and
\begin{equation}\label{Sec4:SOI}
\mathcal{I}^{\Lambda}_{\kappa}[f,g]:=\int_{\lambda_r}^{1}f(x){\rm
e}^{{\rm i}\kappa g(x)}{\rm d}x.
\end{equation}
Note that the function $\varphi_\kappa$ defined as in \eqref{Sec4:integrandSI} has a singularity
at the origin but has no oscillation and the integrand in \eqref{Sec4:SOI} has no singularity
or stationary point but has oscillation. We shall treat these two integrals separately.

We shall  develop a quadrature rule for computing the singular integral \eqref{Sec4:weakSI} having a polynomial order (in terms of the number of nodes used in the partition) of convergence, and a quadrature rule for computing  the oscillatory integral  \eqref{Sec4:SOI} having a polynomial order (in terms of the wave number) of convergence.

We now consider the integral $\mathcal{I}[\varphi_\kappa]$ that appears in \eqref{Sec4:weakSI}.
The integrand $\varphi_\kappa$ defined by \eqref{Sec4:integrandSI} does not oscillate rapidly.
The classical quadrature rules for weakly singular integrals developed in \cite{Xu1} can then be used to treat the singularity. Below, we briefly review the quadrature rules.

We begin with describing the Gauss-Legendre  quadrature rule for integral $\mathcal{I}^{[a,b]}[\psi]:=\int_{a}^{b}\psi(x){\rm d}x$, where $\psi$ is a smooth function defined on $[a,b]$.
Given $m\in\mathbb{N}$, we denote by
$-1<t_1<t_2<\ldots<t_{m-1}<t_{m}<1$
the zeros of the Legendre polynomial $P_{m}$ of degree $m$ and by
$\omega_j:=2(1-t_j^2)[m P_{m-1}(t_j)]^{-2}$ for $j\in\mathbb{Z}_{m}^{+}$,
the weights of the Gauss-Legendre quadrature rule which has the form
$\mathcal{Q}_{m}^{[a,b]}[\psi]:=(b-a)/2\sum_{j\in\mathbb{Z}_{m}^{+}}\omega_j\psi\left([(b-a)t_j+(b+a)]/2\right)$.
There is a constant $\xi\in[a,b]$ such that the error of the approximation $\mathcal{Q}_{m}^{[a,b]}[\psi]$
to $\mathcal{I}^{[a,b]}[\psi]$ is given by
$\mathcal{R}_{m}^{[a,b]}[\psi]:=(b-a)^{2m+1}\psi^{(2m)}(\xi)/(2^{2m}(2m+1)!)$.

We now recall the integral method proposed in \cite{Xu1}. Given $m\in \mathbb{N}$, let $p:=(2m+1)/(1+\mu)$.
For $s\in\mathbb{N}$ with $s>1$, according to the parameter $p$ we choose $s+1$ points given by
$x_j:=s^{-p}j^p$ for $j\in\mathbb{Z}_{s}$.
The quadrature rule  for $\mathcal{I}[\varphi_\kappa]$ is obtained by replacing $\varphi_\kappa$ on $[x_0, x_1]$ by zero and using $\mathcal{Q}_{m}^{[x_j,x_{j+1}]}[\varphi_\kappa]$  for computing the integrals  $\mathcal{I}^{[x_j,x_{j+1}]}[\varphi_\kappa]$ for $j\in\mathbb{Z}_{s-1}^+$.
Integral $\mathcal{I}^{\mu}[f,g]$ defined by \eqref{Sec4:weakSI} is then approximated by the quadrature formula
\begin{equation}\label{alm:Sec4SP}
\mathcal{Q}_{\mu,m}^{s}[f,g]:=\lambda_r\sum_{j\in\mathbb{Z}_{s-1}^+}\mathcal{Q}_{m}^{[x_j,x_{j+1}]}[\varphi_\kappa].
\end{equation}

We next estimate the error
$\mathcal{E}_{\mu,m}^{s}[f,g]:=\abs{\mathcal{Q}_{\mu,m}^{s}[f,g]-\mathcal{I}^{\mu}[f,g]}$.
We need a lemma that estimates the norm of $w_\kappa(x):=\exp\set{u_\kappa(x)}$, for $x\in I$, where
$u_\kappa(x):={\rm i}\kappa g\left(\lambda_rx\right)$ for $x\in I$.
This requires the use of the Fa\`{a} di Bruno formula \cite{Johnson,Riordan, Roman} for derivatives of the composition of two functions. For a fixed $n\in\mathbb{N}$,
if the derivatives of order $n$ of two functions $\phi$ and $\psi$ are defined, then
\begin{equation}\label{Sec4:Faa}
(\phi\circ\psi)^{(n)}=\sum_{j\in\mathbb{Z}_n^+}\phi^{(j)}(\psi)B_{n,j}\left(\psi^{(1)}, \psi^{(2)},\ldots, \psi^{(n-j+1)} \right),
\end{equation}
where  for $j\in\mathbb{Z}_n^+$,
\begin{equation}\label{Sec4:BellP}
B_{n,j}(x_1,x_2,\ldots,x_{n-j+1})=\sum\dfrac{n!}{m_1!m_2!\cdots m_{n-j+1}!} \prod_{l\in\mathbb{Z}_{n-j+1}^+}\left(\dfrac{x_l}{l!}\right)^{m_l},
\end{equation}
where the sum is taken over all $(n-j+1)$-tuples $(m_1$, $\ldots$, $m_{n-j+1})$ satisfying the constraints $\sum_{l\in\mathbb{Z}_{n-j+1}^+}m_l=j$ and
$\sum_{l\in\mathbb{Z}_{n-j+1}^+}lm_l=n$.
For $n\in\mathbb{N}$ and for $j\in\mathbb{Z}_n^+$, we let
$B_{n,j}:=B_{n,j}(1,1,\ldots,1)$ for brevity. Note that
$B_{n,j}=1/j!\sum_{l\in\mathbb{Z}_j}(-1)^{j-l}C_{j}^{l}l^n$
are the Stirling numbers of the second kind \cite{Abramowitz}. They have the property \cite{Rennie} that
for $n\in\mathbb{N}$ with $n\geq 2$ and $j\in\mathbb{Z}_n^+$ satisfying $j\leq 2^{-1}n$,
$B_{n,n-j+1}\leq B_{n,j}$.
For $n\in\mathbb{N}$, the Bell number has the bound \cite{Berend} that
$B_n:=\sum_{j\in\mathbb{Z}_n}B_{n,j}\leq\left(\dfrac{0.792n}{\ln{(n+1)}}\right)^n$.

\begin{lemma}\label{lem:Sec4Singularity}
If $g\in C^{2m}(I)$ for some $m\in\mathbb{N}$ satisfying Assumption $\ref{Sec4:Assume}$, then
there exists a positive constant  $c$  such that for all $\kappa>1$,
$\norm{w_\kappa}_{2m}\leq c$.
\end{lemma}
\begin{proof}
For $n\in\mathbb{Z}_{2m}^+$, by using \eqref{Sec4:Faa}, we obtain that
\begin{equation}\label{Sec4:proof0LemS}
w_\kappa^{(n)}=\sum_{j\in\mathbb{Z}_n^+}w_\kappa B_{n,j}\left(u_\kappa^{(1)}, u_\kappa^{(2)},\ldots, u_\kappa^{(n-j+1)} \right).
\end{equation}
From the definition of $u_\kappa$, we have two inequalities below.
If $j\in\mathbb{Z}_n^+$ with $j\geq r+1$, we have that
\begin{equation}\label{Sec4:proof1LemS}
\abs{u_\kappa^{(j)}(x)}\leq\kappa\lambda_r^j\abs{g^{(j)}\left(\lambda_rx\right)}\leq
\norm{g}_{n}~\text{for} ~x\in I.
\end{equation}
For $j\in\mathbb{Z}_n^+$ with $j< r+1$, by  Assumption \ref{Sec4:Assume} there exists a constant $\xi_x\in[0,\lambda_rx]$ such that
\begin{equation*}
g^{(j)}\left(\lambda_rx\right)=g^{(r+1)}(\xi_x)/(r+1-j)!\left(\lambda_rx\right)^{r+1-j},
~\text{for}~x\in I,
\end{equation*}
which implies that for all $j\in\mathbb{Z}_n^+$ with $j< r+1$ and  $x\in I$,
\begin{equation}\label{Sec4:proof2LemS}
\abs{u_\kappa^{(j)}(x)}\leq\kappa\lambda_r^j\abs{g^{(j)}\left(\lambda_rx\right)}
\leq\norm{g}_{r+1}.
\end{equation}
Letting $M:=\max\set{1,\norm{g}_{r+1},\norm{g}_{2m}}$ and applying \eqref{Sec4:BellP} with \eqref{Sec4:proof1LemS} and \eqref{Sec4:proof2LemS}, we observe for all $\kappa>1$ and $x\in I$ that
$\abs{B_{n,j}\left(u_\kappa^{(1)}, u_\kappa^{(2)},\ldots, u_\kappa^{(n-j+1)} \right)}\leq B_{n,j}M^{j}$.
 This together with \eqref{Sec4:proof0LemS} yields for all $\kappa>1$ and $n\in \mathbb{Z}_{2m}^+$ that
$\norm{w_\kappa^{(n)}}_{\infty}\leq\norm{w_\kappa}_\infty\sum_{j\in\mathbb{Z}_n^+}B_{n,j}M^{j}
\leq\norm{w_\kappa}_\infty B_nM^n$,
which is a constant independent of $\kappa$.  This together with $\norm{w_\kappa}_\infty\leq 1$ yields the desired estimate.
\end{proof}

We next estimate the derivatives of the function $\varphi_\kappa$ defined by \eqref{Sec4:integrandSI}.
We recall the Leibniz formula for the $n$-th derivative of the product of two functions for $n\in\mathbb{N}$,
\begin{equation}\label{Sec4:Leibniz}
\left(\phi\psi\right)^{(n)}=\sum_{m\in\mathbb{Z}_n} C_n^m\phi^{(m)}\psi^{(n-m)},
\end{equation}
where $C_n^m:=\dfrac{n!}{m!(n-m)!}$
are the binomial coefficients that satisfy
\begin{equation}\label{Sec4:BinomialC}
\sum_{m\in\mathbb{Z}_n}C_n^m=2^n.
\end{equation}

\begin{lemma}\label{lem:Sec4SDer0}
Let $m\in\mathbb{N}$.
If $f$ is of ${\rm Type}(\mu,2m,\set{0})$ for some $\mu\in (-1,1)$ and $g\in C^{2m}(I)$ satisfies Assumption $\ref{Sec4:Assume}$, then there exists a positive constant  $c$ such that for all  $\kappa>1$ and  $x\in (0,1]$,
$\abs{\varphi_\kappa^{(2m)}(x)}\leq c\lambda_r^\mu x^{\mu-2m}$.
\end{lemma}
\begin{proof}
Applying the Leibniz formula \eqref{Sec4:Leibniz} to the function $\varphi_\kappa$ yields
\begin{eqnarray}\label{Sec4:proof1LemSDer0}
\varphi_\kappa^{(2m)}(x)=\sum_{j\in\mathbb{Z}_{2m}} C_{2m}^{j}\left(f\left(\lambda_rx\right)\right)^{(j)} w_\kappa^{(2m-j)}(x),~\text{for}~x\in (0,1].
\end{eqnarray}
From the assumption on $f$, there exists a positive constant $c$
such that for all  $\kappa>1$,  $j\in\mathbb{Z}_{2m}$ and  $x\in (0,1]$,
$\abs{\left(f\left(\lambda_rx\right)\right)^{(j)}}\leq \lambda_r^j\abs{f^{(j)}\left(\lambda_rx\right)}
\leq c\lambda_r^\mu x^{\mu-j}$.
This together with \eqref{Sec4:proof1LemSDer0} yields a positive
constant  $c$ such that for all  $\kappa>1$ and  $x\in (0,1]$,
\begin{eqnarray*}
\abs{\varphi_\kappa^{(2m)}(x)}
\leq c\sum_{j\in\mathbb{Z}_{2m}} C_{2m}^{j}\lambda_r^\mu x^{\mu-j}\abs{w_\kappa^{(2m-j)}(x)}
\leq c\lambda_r^\mu x^{\mu-2m}\norm{w_\kappa}_{2m}\sum_{j\in\mathbb{Z}_{2m}} C_{2m}^{j}.
\end{eqnarray*}
Using Lemma \ref{lem:Sec4Singularity} and formula \eqref{Sec4:BinomialC} in the inequality above, we obtain the desired estimate.
\end{proof}

We need a technical result for the integral of a function  of ${\rm Type}(\mu,0,\{0\})$ for some $\mu\in (-1,1)$.

\begin{lemma}\label{lem:Sec4S}
If $\phi$ is of ${\rm Type}(\mu,0,\{0\})$ for some  $\mu\in (-1,1)$, then there exists
a positive constant $c$  such that for all $0<b<1$ and $\varrho>0$,
$\int_{0}^b\abs{\phi(\varrho x)} {\rm d}x\leq{c\varrho^\mu}b^{1+\mu}$.
\end{lemma}
\begin{proof}
We prove this result by  bounding $\abs{\phi(\varrho x)}$ by  $c\varrho^\mu x^\mu$ and  computing the resulting integral exactly.
\end{proof}

For a quadrature formula $F$ for approximation of \eqref{Sec4:weakSI}, we use $\mathcal{N}(F)$ to denote  the number of evaluations of the integrand $\varphi_\kappa$ used in the formula.
With the above preparation, we estimate the error $\mathcal{E}_{\mu,m}^{s}[f,g]$ and $\mathcal{N}\left(\mathcal{Q}_{\mu,m}^{s}[f,g]\right)$ in the following theorem.

\begin{theorem}\label{thm:Sec4SP}Let $m\in\mathbb{N}$.
If $f$ is of ${\rm Type}(\mu,2m,\set{0})$ for some $\mu\in (-1,1)$ and $g\in C^{2m}(I)$ satisfies Assumption $\ref{Sec4:Assume}$, then there exists a positive constant $c$ such that for all $\kappa>1$, $s\in \mathbb{N}$ with $s>1$
\begin{equation}\label{Sec4:ErrorSP}
\mathcal{E}_{\mu,m}^{s}[f,g]\leq c\kappa_{\sigma(r)}^{-(1+\mu)/(r+1)}s^{-2m},
\end{equation}
and $\mathcal{N}\left(\mathcal{Q}_{\mu,m}^{s}[f,g]\right)\leq (s-1)m$.
If $g(x)=x$ for $x\in I$, then the upper bound in \eqref{Sec4:ErrorSP} reduces to $c\kappa^{-1-\mu}s^{-2m}$.
\end{theorem}
\begin{proof}
Let $\mathcal{E}_{0}[\varphi_\kappa]:=\abs{\int_{x_0}^{x_1}\varphi_\kappa(x){\rm d}x}$ and
$\mathcal{E}_{j}[\varphi_\kappa]:=\abs{\mathcal{Q}_{m}^{[x_j,x_{j+1}]}[\varphi_{\kappa}]
-\mathcal{I}^{[x_j,x_{j+1}]}[\varphi_{\kappa}]}$
for $j\in\mathbb{Z}_{s-1}^{+}$.
We prove  \eqref{Sec4:ErrorSP} by estimating $\mathcal{E}_{j}[\varphi_\kappa]$
and then sum them over $j\in\mathbb{Z}_{s-1}$.

We first consider $\mathcal{E}_{0}[\varphi_\kappa]$.
By  applying Lemma \ref{lem:Sec4S}
with $\varrho:=\lambda_r$ and $b:=x_1=s^{-(2m+1)/(1+\mu)}$
we conclude that there exists a positive constant $c$ such that for all $\kappa>1$ and
$s\in\mathbb{N}$ with $s>1$
\begin{eqnarray*}
\mathcal{E}_{0}[\varphi_\kappa]
\leq\int_{0}^{x_1}\abs{f\left(\lambda_r x\right)}{\rm d}x
\leq c\lambda_r^\mu s^{-2m-1}.
\end{eqnarray*}
Note that $\lambda_r$ is defined by \eqref{Sec4:kappa}.
For $j\in\mathbb{Z}_{s-1}^{+}$, by the error bound of the Gauss-Legendre quadrature,
there exists a constant $\xi_j\in[x_j,x_{j+1}]$ such that
$\mathcal{E}_{j}[\varphi_\kappa]\leq h_j^{2m+1}\abs{\varphi_\kappa^{(2m)}(\xi_j)}/(2^{2m}(2m+1)!)$,
where $h_j:=x_{j+1}-x_j$.
Using  Lemma \ref{lem:Sec4SDer0} with  $\xi_j\geq x_j$ for $j\in\mathbb{Z}_{s-1}^{+}$,
there exists a positive constant $c$  such that for all $\kappa>1$,
$s\in\mathbb{N}$ with $s>1$ and $j\in\mathbb{Z}_{s-1}^{+}$,
$\abs{\varphi_\kappa^{(2m)}(\xi_j)}\leq c\lambda_r^\mu x_j^{\mu-2m}$.
Substituting this inequality into  the inequality above and employing the estimate
$h_j^{2m+1}x_j^{\mu-2m}\leq cs^{-2m-1}$ obtained from the proof of Theorem 2.3 in \cite{Xu1},
we see that there exists a positive constant $c$ such that for all $\kappa>1$,
$s\in \mathbb{N}$ with $s>1$ and $j\in\mathbb{Z}_{s-1}^{+}$,
$\mathcal{E}_{j}[\varphi_\kappa]\leq c \lambda_r^\mu s^{-2m-1}$.
Summing up the  bound of errors $\mathcal{E}_{j}[\varphi_\kappa]$ over $j\in\mathbb{Z}_{s-1}$ and using the definition \eqref{Sec4:kappa} of $\lambda_r$  we obtain the estimate \eqref{Sec4:ErrorSP}.

The bound on the number of functional evaluations used in  the algorithm \eqref{alm:Sec4SP} may be obtained directly.

When $g(x)=x$, we substitute $r=0$ and  $\kappa_{\sigma(0)}=\kappa$ into estimate \eqref{Sec4:ErrorSP} to yield the special result.
\end{proof}

We next develop a quadrature formula for computing the oscillatory integral \eqref{Sec4:SOI}.
For $n\in\mathbb{N}$, we choose the  partition
$\Pi_{\kappa,\sigma(r)}$  of $\Lambda$ with nodes defined by
\begin{equation}\label{Sec4:PartitionOI}
x_j:=\kappa_{\sigma(r)}^{(j/n-1)/(r+1)},\ \ j\in\mathbb{Z}_n.
\end{equation}
According to the partition, the integral \eqref{Sec4:SOI} is rewritten as
$\mathcal{I}_{\kappa}^{\Lambda}[f,g]=\sum_{j\in\mathbb{Z}_n^+}\mathcal{I}_{\kappa}^{[x_{j-1},x_{j}]}[f,g]$.
For each $j\in\mathbb{Z}_{n}^+$, we use quadrature formula \eqref{alm:Sec2Filon} to compute an approximation $\mathcal{Q}_{N_j,\kappa,m}^{[x_{j-1},x_{j}]}[f,g]$ of the integral $\mathcal{I}_{\kappa}^{[x_{j-1},x_{j}]}[f,g]$, where $N_j$ is a positive integer to be specified later.
We estimate the error $\mathcal{E}_j[f,g]:=\abs{\mathcal{I}_{\kappa}^{[x_{j-1},x_{j}]}[f,g]
-\mathcal{Q}_{N_j,\kappa,m}^{[x_{j-1},x_{j}]}[f,g]}$ for $j\in\mathbb{Z}_{n}^+$ in the following lemma. To this end, we define $\Theta(y):=\Psi(g(1)y)$ for $y\in I$,
$\delta(r):=\min\set{\abs{g^{(r+1)}(x)}/(r+1)!: x\in I}$,
$\beta_n(r):=\lambda_r^{-1/n}-1$ with $\lambda_r$ defined by \eqref{Sec4:kappa},
$M_j:=\max\set{\abs{g'(x_{j-1})},\abs{g'(x_j)}}$ and
$q_j:=M_j x_{j-1}/g(x_{j-1})$, for $j\in\mathbb{Z}_n^+$.

\begin{lemma}\label{lem:Sec4SubE}
If $\Theta$ is of ${\rm Type}(\alpha,m+1,\set{0})$ for some $\alpha\in (-1,1)$ and some $m\in\mathbb{N}$,
and $g$ satisfies Assumption $\ref{Sec4:Assume}$, then there exists a positive  constant $c$ such that for all $\kappa>1$, $n\in\mathbb{N}$ and $j\in\mathbb{Z}_n^+$
\begin{eqnarray*}
\mathcal{E}_j[f,g]\leq c\dfrac{\abs{g(1)}^{-\alpha}(\delta(r))^{\alpha-1}}{(m-1)!\kappa^2N_j^{m-1}}q_j^mx_{j-1}^{(\alpha-1)(r+1)}\left(\beta_n(r)\right)^m.
\end{eqnarray*}
\end{lemma}
\begin{proof}
We prove this result  by applying  \eqref{Sec2:FilonM1} in Theorem \ref{thm:Sec2FilonM} to $\mathcal{E}_j[f,g]$
for  $j\in\mathbb{Z}_{n}^+$. From \eqref{Sec2:FilonM1}, we have for all $j\in\mathbb{Z}_{n}^+$ that
$\mathcal{E}_j[f,g]\leq\dfrac{3(m+1)}{m!\kappa^2N_j^{m-1}}\norm{\Psi^{(m+1)}}_{\infty}M_j^mh_j^m$,
where $h_j:=x_{j}-x_{j-1}$. In the estimate above, using the relation
$\Psi^{(m+1)}(g(1)y)= (g(1))^{-m-1}\Theta^{(m+1)}(y)$ for $y\in(0,1]$ with the assumption on $\Theta$, we observe that there exists a positive constant $c$ such that for all $\kappa>1$,
$n\in\mathbb{N}$ and $j\in\mathbb{Z}_{n}^+$,
\begin{eqnarray*}
\mathcal{E}_j[f,g]
\leq c\dfrac{\abs{g(1)}^{-\alpha}}{(m-1)!\kappa^2N_j^{m-1}}q_j^m{(g(x_{j-1}))^{\alpha-1}}x_{j-1}^{-m}h_j^m.
\end{eqnarray*}
Note that
$(g(x_{j-1}))^{\alpha-1}\leq(\delta(r))^{\alpha-1} x_{j-1}^{(r+1)(\alpha-1)}$ and $x_{j-1}^{-m}h_j^m=\left(\beta_n(r)\right)^m$
for $j\in\mathbb{Z}_{n}^+$. Substituting these relations into the inequality above yields the desired result.
\end{proof}

We estimate an upper bound of the quantity $\mathbb{Q}_n:=\max_{j\in\mathbb{Z}_{n}^+}\set{q_j}$.

\begin{lemma}\label{lem:Sec4Bound}
If $g$ satisfies Assumption $\ref{Sec4:Assume}$, then for all $n\in\mathbb{N}$ and $\kappa>1$,
$\mathbb{Q}_n\leq (r+1)\kappa_{\sigma(r)}^{r/(n(r+1))}{\sigma(r)}/{\delta(r)}$.
\end{lemma}
\begin{proof}
Using the definition of $\delta(r)$ and ${\sigma(r)}$,
we have that $g(x)\geq\delta(r) x^{r+1}$ and $\abs{g'(x)}\leq(r+1)\sigma(r) x^r$ for $x\in I$.
Thus,  $\abs{g(x_{j-1})}\geq \delta(r) x_{j-1}^{r+1}$ and $M_j\leq (r+1)\sigma(r)  x_{j}^{r}$.
For $n\in\mathbb{N}$, we obtain for $\kappa>1$ and $j\in\mathbb{Z}_{n}^+$ that
$q_j\leq (r+1)\kappa_{\sigma(r)}^{r/(n(r+1))}{\sigma(r)}/{\delta(r)}$.
\end{proof}

We now discuss the choice of $N_j$.
Lemma \ref{lem:Sec4Bound} demonstrates that the upper bound of $\mathbb{Q}_n$ is independent
of $\kappa$ for the case $r=0$ and it depends on $\kappa$ for the case $r>0$.
Therefore, we need to consider these two cases separately.
For the case $r=0$, we choose
\begin{equation}\label{Sec4:NP0}
 N_j:=\left\lceil q_j\right\rceil~\text{for}~j\in\mathbb{Z}_{n}^+.
\end{equation}
For the case $r>0$, we choose
\begin{equation}\label{Sec4:NP}
 N_j:=\left\lceil q_j^{m/(m-1)}\right\rceil~\text{for}~j\in\mathbb{Z}_{n}^+,
\end{equation}
so that $N_j^{-(m-1)}q_j^m$ for all $j\in\mathbb{Z}_{n}^+$ are independent of $\kappa$.
With this choice of $N_j$, we use $\mathcal{Q}_{N_j,\kappa,m}^{[x_{j-1},x_{j}]}[f,g]$ to approximate the
integral $\mathcal{I}_{\kappa}^{[x_{j-1},x_{j}]}[f,g]$ for $j\in\mathbb{Z}_{n}^+$. Thus,
integral $\mathcal{I}_\kappa^\Lambda[f,g]$ defined by  \eqref{Sec4:SOI}
is approximated by the quadrature formula
\begin{equation}\label{alm:Sec4P}
\mathcal{Q}^\Lambda_{\kappa,n,m}[f,g]:=\sum_{j\in\mathbb{Z}_n^+}\mathcal{Q}_{N_j,\kappa,m}^{[x_{j-1},x_{j}]}[f,g].
\end{equation}

We next estimate the error
$\mathcal{E}_{\kappa,n,m}^{\Lambda}[f,g]:=
\abs{\mathcal{I}_{\kappa}^{\Lambda}[f,g]-\mathcal{Q}_{\kappa,n,m}^{\Lambda}[f,g]}$ and
$\mathcal{N}\left(\mathcal{Q}_{\kappa,n,m}^{\Lambda}[f,g]\right)$.
We first consider the case $r=0$. For simplicity of notation,
we let $\delta_0:=\delta(0)$, $\sigma_0:=\sigma(0)$ and $\beta_n^0:=\beta_n(0)$.

\begin{theorem}\label{thm:Sec4P0}
Let $r=0$. If $\Theta$ is of ${\rm Type}(\alpha,m+1,\set{0})$ for some  $\alpha\in (-1,1)$ and some $m\in\mathbb{N}$ and $g$ satisfies Assumption $\ref{Sec4:Assume}$ with $r=0$,
then there exists a positive constant $c$ such that for all  $\kappa>1$ and $n\in\mathbb{N}$ satisfying $n\geq\ln{\kappa_{\sigma_0}}/\ln{2}$,
\begin{equation}\label{Sec4:ErrorP0}
\mathcal{E}_{\kappa,n,m}^{\Lambda}[f,g]\leq c \abs{g(1)}^{-\alpha} \delta_0^{\alpha-2}\sigma_0\kappa^{-2}\kappa_{\sigma_0}^{1-\alpha}{\ln{\kappa_{\sigma_0}}}
\left(\beta_n^0\right)^{m-1}.
\end{equation}
There holds the estimate
$\mathcal{N}\left(\mathcal{Q}_{\kappa,n,m}^{\Lambda}[f,g]\right)\leq\lceil \sigma_0/\delta_0\rceil nm+1$.
\end{theorem}
\begin{proof}
We prove \eqref{Sec4:ErrorP0} by  estimating $\mathcal{E}_{j}[f,g]$ for $j\in\mathbb{Z}_{n}^+,$ and then summing them  over $j$. By employing Lemmas \ref{lem:Sec4SubE} and \ref{lem:Sec4Bound} with $r=0$ and the choice \eqref{Sec4:NP0} of $N_j$, there exists a positive constant $c$ such that for
all $\kappa>1$,  $n\in\mathbb{N}$ and $j\in\mathbb{Z}_{n}^+$,
$\mathcal{E}_j[f,g]\leq c \abs{g(1)}^{-\alpha}\delta_0^{\alpha-2}\sigma_0\kappa^{-2}{x_{j-1}^{\alpha-1}}
\left(\beta_n^0\right)^{m}$.
Since $-1<\alpha<1$, we see that
$x_{j-1}^{\alpha-1}\leq x_{0}^{\alpha-1}= \kappa_{\sigma_0}^{1-\alpha}$.
Substituting this inequality into the inequality above yields  that there exists a positive constant $c$ such that for all $\kappa>1$,  $n\in\mathbb{N}$ and $j\in\mathbb{Z}_{n}^+$,
$\mathcal{E}_j[f,g]\leq c \abs{g(1)}^{-\alpha}\delta_0^{\alpha-2}\sigma_0\kappa^{-2}\kappa_{\sigma_0}^{1-\alpha}\left(\beta_n^0\right)^m$.
Moreover, condition $n\geq\ln{\kappa_{\sigma_0}}/\ln{2}$ implies that $n\beta_n^0\leq2\ln{\kappa_{\sigma_0}}$.
Thus, there exists a positive constant $c$ such that for all  $\kappa>1$ and
$n\in\mathbb{N}$ satisfying $n\geq\ln{\kappa_{\sigma_0}}/\ln{2}$,
\begin{equation*}
\mathcal{E}_{\kappa,n,m}^{\Lambda}[f,g]\leq\sum_{j\in\mathbb{Z}_{n}^+}\mathcal{E}_j[f,g]\leq c\abs{g(1)}^{-\alpha}\delta_0^{\alpha-2}\sigma_0\kappa^{-2}\kappa_{\sigma_0}^{1-\alpha}\ln{\kappa_{\sigma_0}}
\left(\beta_n^0\right)^{m-1},
\end{equation*}
which is the desired estimate \eqref{Sec4:ErrorP0}.

The estimate of $\mathcal{N}\left(\mathcal{Q}_{\kappa,n,m}^{\Lambda}[f,g]\right)$ is obtained by using  formula \eqref{alm:Sec4P}, the choice of $N_j$ and Lemma \ref{lem:Sec4Bound} with $r=0$.
\end{proof}

In passing, we comment on the hypothesis imposed to $\Theta$ in the last theorem. If $f\in C^m(I)$ for some $m\in\mathbb{N}$ and $g\in C^{m}(I)$ satisfies Assumption \ref{Sec4:Assume}, then $\Theta$ is of ${\rm Type}(-r/(r+1),m,\set{0})$. A proof of this conclusion may be found in \cite{Graham}.
As a direct consequence of Theorem \ref{thm:Sec4P0}, we have the next estimates for the case with
the linear oscillator $g(x)=x$ for $x\in I$, where we have that $g(1)=\delta_0=\sigma_0=1$ and $\kappa_{\sigma_0}=\kappa$.

\begin{corollary}\label{cor:Sec4LinearP}
If $f$ is of ${\rm Type}(\alpha,m+1,\set{0})$ for some  $\alpha\in (-1,1)$ and some $m\in\mathbb{N}$, then
there exists a positive constant $c$ such that for all  $\kappa>1$ and $n\in\mathbb{N}$ satisfying
$n\geq\ln{\kappa}/\ln{2}$,
$\mathcal{E}_{\kappa,n,m}^{\Lambda}[f,g]\leq c{\kappa^{-\alpha-1}}{\ln{\kappa}}\theta_n^{m-1}$,
where $\theta_n:=\kappa^{1/n}-1$. For $n\in\mathbb{N}$, there holds the estimate $\mathcal{N}\left(\mathcal{Q}_{\kappa,n,m}^{\Lambda}[f,g]\right)\leq nm+1$.
\end{corollary}

We now consider the case $r>0$.

\begin{theorem}\label{thm:Sec4P}
Let $r>0$. If  $\Theta$ is of ${\rm Type}(\alpha,m+1,\set{0})$ for some  $\alpha\in (-1,1)$ and some $m\in\mathbb{N}$ and $g$ satisfies Assumption $\ref{Sec4:Assume}$, then there exists a positive constant $c$ such that for all  $\kappa>1$ and  $n\in\mathbb{N}$ satisfying $n\geq\ln{\kappa_{\sigma(r)}}/\ln{2}$,
\begin{equation}\label{Sec4:ErrorP}
\mathcal{E}_{\kappa,n,m}^{\Lambda}[f,g]\leq c (r+1) \abs{g(1)}^{-\alpha}(\delta(r))^{\alpha-1}\kappa^{-2}
\kappa_{\sigma(r)}^{1-\alpha}{\ln{\kappa_{\sigma(r)}}}(\beta_n(r))^{m-1}.
\end{equation}
There holds
$\mathcal{N}\left(\mathcal{Q}_{\kappa,n,m}^{\Lambda}[f,g]\right)
 \leq\left\lceil\big{(}2(r+1)\sigma(r)/\delta(r)\big{)}^{m/(m-1)}\right\rceil nm+1$.
\end{theorem}
\begin{proof}
Estimate \eqref{Sec4:ErrorP} is proved in a similar way to that of estimate \eqref{Sec4:ErrorP0} in Theorem \ref{thm:Sec4P0}. The only difference is that instead of using the hypotheses and results for the special case $r=0$, we use those for the case  $r>0$.
We leave the details to the interested reader.

It remains to estimate $\mathcal{N}\left(\mathcal{Q}_{\kappa,n,m}^{\Lambda}[f,g]\right)$.
According to formula \eqref{alm:Sec4P} and Lemma \ref{lem:Sec4Bound}, we obtain that
$\mathcal{N}\left(\mathcal{Q}_{\kappa,n,m}^{\Lambda}[f,g]\right)
 \leq\sum_{j\in\mathbb{Z}_n^+}\set{N_jm+1}-(n-1)
 \leq\left\lceil\left((r+1)\kappa_{\sigma(r)}^{1/n}\sigma(r)/\delta(r)\right)^{m/(m-1)}\right\rceil nm+1$.
Note that condition  $n\geq\ln{\kappa_{\sigma(r)}}/\ln{2}$ for  $n\in\mathbb{N}$ implies that $\kappa_{\sigma(r)}^{1/n}\leq2$. This concludes the last result.
\end{proof}

To close this section, we point it out that the oscillatory integral \eqref{Sec2:OI0} may be computed according to the formula $\mathcal{Q}_{\kappa,n,m}^{s,\widetilde{m}}[f,g]:=\mathcal{Q}_{\mu,\widetilde{m}}^{s}[f,g]
+\mathcal{Q}_{\kappa,n,m}^{\Lambda}[f,g]$, with two fixed positive integers $\widetilde{m}$ and $m$.

\section{Quadratures with an Exponential Order of Convergence}

In this section, we develop a quadrature method for computing the oscillatory integral \eqref{Sec2:OI0} having an exponential order of convergence. As in section 4, we shall write integral \eqref{Sec2:OI0} as the sum of a weakly singular integral \eqref{Sec4:weakSI} and an oscillatory integral \eqref{Sec4:SOI}, and treat them separately. Specifically, we shall develop a quadrature rule for computing \eqref{Sec4:weakSI} having an exponential order in terms of a constant $\gamma\in (0,1)$ of convergence and a quadrature rule for computing \eqref{Sec4:SOI} having an exponential order in terms of the wave number of convergence.

We now describe the quadrature rule for the  integral  $\mathcal{I}[\varphi_\kappa]$ that appears in \eqref{Sec4:weakSI}. In developing the quadrature rule for computing  $\mathcal{I}[\varphi_\kappa]$ we borrow an idea from \cite{Schwab} (see also \cite{Xu2}). We begin with describing a partition of $I$. For $\gamma\in (0,1)$ and $s\in \mathbb{N}$ with $s>1$, let $\Pi^\gamma$ be a partition of $I$ with nodes defined by $x_0=0$ and $x_j=\gamma^{s-j}$ for $j\in\mathbb{Z}_{s}^+$.
For $\varepsilon\geq 1$, we choose
$m_j:=\lceil j\varepsilon\rceil$ for $j\in\mathbb{Z}_{s-1}^+$.
The quadrature rule for computing $\mathcal{I}[\varphi_\kappa]$ is formed by placing $\varphi_\kappa$ on
$[x_0, x_1]$ by 0 and using $\mathcal{Q}_{m_j}^{[x_j,x_{j+1}]}[\varphi_\kappa]$ for computing the integrals $\mathcal{I}^{[x_j,x_{j+1}]}[\varphi_\kappa]$, for $j\in\mathbb{Z}_{s-1}^+$.
Integral $\mathcal{I}^\mu[f,g]$ defined by \eqref{Sec4:weakSI} is then approximated by the quadrature formula
\begin{equation}\label{alm:Sec5SE}
 \mathcal{Q}_{\gamma}^{s}[f,g]:=\lambda_r\sum_{j\in\mathbb{Z}_{s-1}^+}\mathcal{Q}_{m_j}^{[x_j,x_{j+1}]}[\varphi_\kappa],
\end{equation}
where $\lambda_r$ is defined by \eqref{Sec4:kappa}.
We next estimate the error $\mathcal{E}_{\gamma}^{s}[f,g]:=\abs{\mathcal{I}_{\mu}[f,g]-\mathcal{Q}_{\gamma}^s[f,g]}$.
To this end, we impose the following hypothesis on $g$.

\begin{assumption} \label{Sec5:Assume}
There  exists a positive constant $\zeta>1$ such that for all $n\in\mathbb{N}$, $\norm{g}_n\leq \zeta$.
\end{assumption}

We first estimate the norm of $w_\kappa(x)=\exp\set{u_\kappa(x)}$ with
$u_\kappa(x)={\rm i}\kappa g\left(\lambda_r x\right)$ for $x\in I$ as Section 4.

\begin{lemma}\label{lem:Sec5wNorm}
If $g\in C^{\infty}(I)$ satisfies Assumption $\ref{Sec5:Assume}$, then  for all $\kappa>1$ and $n\in\mathbb{N}$, $\norm{w_\kappa}_n\leq  B_n\zeta^{n}$.
\end{lemma}

\begin{proof}
  The proof of this lemma is similar to that of Lemma \ref{lem:Sec4Singularity}.
  Using Assumption \ref{Sec5:Assume} and applying    \eqref{Sec4:BellP} with \eqref{Sec4:proof1LemS} and \eqref{Sec4:proof2LemS} we observe that  for all
  $n\in\mathbb{N}$, $\kappa>1$,  $j\in\mathbb{Z}_n^+$, $l\in\mathbb{Z}_j^+$ and $x\in I$,
  $$\abs{B_{j,l}\left(u_\kappa^{(1)}, u_\kappa^{(2)},\ldots, u_\kappa^{(j-l+1)} \right)}\leq B_{j,l}\zeta^{l}.$$
  This together with
  $w_\kappa^{(j)}=\sum_{l\in\mathbb{Z}_j^+}w_\kappa B_{j,l}\left(u_\kappa^{(1)}, u_\kappa^{(2)},\ldots, u_\kappa^{(j-l+1)} \right)$
  yields that for all $n\in\mathbb{N}$, $j\in\mathbb{Z}_n^+$
  and  $\kappa>1$
  $\norm{w_\kappa^{(j)}}_{\infty}\leq \norm{w_\kappa}_\infty\sum_{l\in\mathbb{Z}_j^+}B_{j,l}\zeta^{l} \leq B_j\zeta^{j} $, since $\norm{w_\kappa}_\infty\leq 1$.
  Note that $B_n<B_{n+1}$ for $n\in\mathbb{N}$ from the  recurrence relation of the Bell number involving binomial coefficients \cite{Wilf}. This together with the inequality above yields the conclusion.
\end{proof}

In the next lemma,
we study a property of the derivatives of
$\varphi_\kappa$ defined as in \eqref{Sec4:integrandSI}.

\begin{lemma}\label{lem:Sec5SDer}
If $f$ is of ${\rm Type}(\mu,\infty,\set{0})$ for some $\mu\in(-1,1)$ and $g\in C^{\infty}(I)$ satisfies  Assumption  $\ref{Sec5:Assume}$, then there exist  two positive constants  $c$  and $\zeta>1$ such that for all $\kappa>1$, $n\in\mathbb{N}$ and $x\in (0,1]$,
$\abs{\varphi_\kappa^{(n)}(x)}\leq c{2^n}B_n\zeta^{n}\lambda_r^\mu x^{\mu-n}$.
\end{lemma}

\begin{proof}
Applying the Leibniz formula \eqref{Sec4:Leibniz} to $\varphi_\kappa$ yields
$\varphi_\kappa^{(n)}(x)
 =\sum_{j\in\mathbb{Z}_{n}}  C_{n}^{j}\left(f\left(\lambda_rx\right)\right)^{(j)}w_\kappa^{(n-j)}(x)$
for $x\in (0,1]$ and $n\in\mathbb{N}$.
By the assumption  on  $f$, there exists a positive constant $c$
such that for all $\kappa>1$,  $n\in\mathbb{N}$ and $x\in (0,1]$,
$\abs{\left(f\left(\lambda_rx\right)\right)^{(n)}}\leq c\lambda_r^\mu x^{\mu-n}$.
Using Lemma \eqref{lem:Sec5wNorm} with the  inequality above, we conclude  that
there exist two positive constants  $c$ and $\zeta>1$ such that for all
$\kappa>1$, $n\in\mathbb{N}$ and  $x\in (0,1]$,
$\abs{\varphi_\kappa^{(n)}(x)}\leq c B_n\zeta^{n}\lambda_r^\mu x^{\mu-n}\sum_{j\in\mathbb{Z}_n}C_{n}^{j}$.
The desired estimate of this lemma is then obtained from this inequality with \eqref{Sec4:BinomialC}.
\end{proof}

We need  two lemmas regarding the parameters $m_j$ for $j\in\mathbb{Z}_{s-1}^+$.

\begin{lemma}\label{lem:Sec5OrderSE1}
Let $s\in\mathbb{N}$ with $s>1$ and $\varepsilon\geq 1$.  If $j$, $l\in\mathbb{Z}_{s-1}^+$, $j\neq l$, then
$m_j\neq m_l$. Moreover, there holds  $(1+\mu)(j-1)< 2m_j-2$ for $j\in\mathbb{Z}_{s-1}^+$.
\end{lemma}
\begin{proof}
We prove the first result by contradiction. Assume to the contrary that there exist $j$,
$l\in\mathbb{Z}_{s-1}^+$ with $j\neq l$ such that $m_j=m_l$. Without loss of generality, we assume $j>l$.
Then, there exists a positive integer $k$ such that $j=l+k$.  We then observe that $(l+k)\varepsilon=j\varepsilon\leq m_j=m_l<l\varepsilon+1$,
which leads to $\varepsilon<k^{-1}\leq1$. This contradicts the assumption $\varepsilon\geq1$ and thus
$m_j\neq m_l$.

The second inequality follows from the assumptions that  $-1<\mu<1$ and $\varepsilon\geq 1$, from which we conclude that $(1+\mu)(j-1)< 2(j-1)\leq2j\varepsilon-2\leq 2m_j-2$.
\end{proof}

\begin{lemma}\label{lem:Sec5OrderSE2}
If $\gamma\in(0,1)$, $\varepsilon\geq 1$ and $\zeta>1$, then there exists a positive constant $c$ such that for all $s\in\mathbb{N}$ with $s>1$ and $j\in\mathbb{Z}_{s-1}^+$,
$B_{2m_j+1}\nu^{(1+\mu)(j-1)+(2m_j+1)}/(2m_j+1)!\leq c$, where $\nu:=\zeta/\gamma$.
\end{lemma}
\begin{proof}
The proof is similar to that of Lemma 4.1 in \cite{Xu2}.
Using the second result of Lemma \ref{lem:Sec5OrderSE1} and the fact $\zeta>\gamma$, we obtain that
$\nu^{(1+\mu)(j-1)+(2m_j+1)}< \nu^{4m_j-1}$.
Applying inequality \eqref{Sec3:Stirling-coro} with $n:=2m_j+1$, we find that
$1/(2m_j+1)!\leq 1/\sqrt{2\pi{\rm e}}\left({\rm e}/(2m_j+1)\right)^{2m_j+3/2}$.
Using the bound of Bell number
$B_n\leq\left(0.792n/\ln{(n+1)}\right)^n$
with $n:=2m_j+1$, we have that
$B_{2m_j+1}\leq\left(0.792(2m_j+1)/\ln{(2m_j+2)}\right)^{2m_j+1}$.
Combining these three inequalities yields
$B_{2m_j+1}\nu^{(1+\mu)(j-1)+(2m_j+1)}/(2m_j+1)!\leq d_j^{2m_j+1}/(\sqrt{2\pi(2m_j+1)} \nu^{3})$,
where $d_j:=0.792\nu^2{\rm e}/\ln{(2m_j+2)}$.
It suffices to prove that there exists a positive constant $c$ such that for all $s\in\mathbb{N}$ with $s>1$ and $j\in\mathbb{Z}_{s-1}^+$,
$d_j^{2m_j+1}\leq c$.
According to the definition of $d_j$ and $m_j$, we observe that
$d_j\leq 0.792\nu^2{\rm e}/\ln(2j\varepsilon+2)$.
It follows for $j\in\mathbb{N}$ satisfying
$j\geq(2\varepsilon)^{-1}({\rm e}^{0.792\nu^2 \rm e}-2)$ that $d_j^{2m_j+1}\leq 1$.
On the other hand, for $j\in\mathbb{N}$ satisfying $j<(2\varepsilon)^{-1}({\rm e}^{0.792\nu^2 \rm e}-2)$,
we have that
$d_j^{2m_j+1}\leq \max\set{d_j^{2m_j+1}: j\in\mathbb{N}, j<(2\varepsilon)^{-1}({\rm e}^{0.792\nu^2 \rm e}-2)}$, which is a constant.
This proves the desired inequality.
\end{proof}

We are now ready to estimate  the error $\mathcal{E}_{\gamma}^{s}[f,g]$ and $\mathcal{N}\left(\mathcal{Q}_{\gamma}^{s}[f,g]\right)$ of the quadrature rule $\mathcal{Q}_{\gamma}^{s}[f,g]$.

\begin{theorem}\label{thm:Sec5SE}
Let $\gamma\in (0,1)$ and $\varepsilon\geq1$. If $f$ is of ${\rm Type}(\mu,\infty,\set{0})$ for some $\mu\in(-1,1)$ and $g\in C^{\infty}(I)$ satisfies  Assumption  $\ref{Sec5:Assume}$,
then there exists a positive constant $c$ such that for all  $\kappa>1$ and integers $s>1$
\begin{equation}\label{Sec5:ErrorSE}
\mathcal{E}_{\gamma}^{s}[f,g]\leq c\kappa_{\sigma(r)}^{-(1+\mu)/(r+1)}\gamma^{(1+\mu)(s-1)-1},
\end{equation}
and $\mathcal{N}\left(\mathcal{Q}_{\gamma}^{s}[f,g]\right)\leq \lceil\varepsilon\rceil(s^2-s)$.
In particular, if $g(x)=x$ for $x\in I$, then the upper bound in \eqref{Sec5:ErrorSE} reduces to
$c\kappa^{-1-\mu}\gamma^{(1+\mu)(s-1)-1}$.
\end{theorem}
\begin{proof}
We prove \eqref{Sec5:ErrorSE} by estimating
$\mathcal{E}_{0}[\varphi_\kappa]$ and $\mathcal{E}_{j}[\varphi_\kappa]:=
\abs{\mathcal{I}^{[x_j,x_{j+1}]}[\varphi_{\kappa}]-\mathcal{Q}_{m_j}^{[x_j,x_{j+1}]}[\varphi_{\kappa}]}$,
for $j\in\mathbb{Z}_{s-1}^{+}$, and then summing them over $j$.

We first consider $\mathcal{E}_{0}[\varphi_\kappa]$.
By applying Lemma \ref{lem:Sec4S} with $\varrho:=\lambda_r$ and $b:=x_1=\gamma^{s-1}$,
there exists a positive constant $c$ such that
for all $\kappa>1$ and $s\in\mathbb{N}$ with $s>1$
\begin{eqnarray*}
\mathcal{E}_{0}[\varphi_\kappa]
\leq\int_{0}^{x_1}\abs{f\left(\lambda_rx\right)}{\rm d}x
\leq c\lambda_r^\mu\gamma^{(1+\mu)(s-1)}.
\end{eqnarray*}
For $j\in\mathbb{Z}_{s-1}^{+}$, by the error bound of the Gauss-Legendre quadrature,
there exists a constant $\xi_j\in[x_j,x_{j+1}]$ such that
$\mathcal{E}_{j}[\varphi_\kappa]\leq h_j^{2m_j+1}/(2^{2m_j}(2m_j+1)!)\abs{\varphi_\kappa^{(2m_j)}(\xi_j)}$,
where $h_j:=x_{j+1}-x_j$. Applying   Lemma \ref{lem:Sec5SDer}  with  $\xi_j\geq x_j$ for $j\in\mathbb{Z}_{s-1}^{+}$, there
exist two positive constants $c$ and $\zeta>1$ such that for all  $\kappa>1$, $s\in\mathbb{N}$ with $s>1$ and $j\in\mathbb{Z}_{s-1}^{+}$,
$\abs{\varphi_\kappa^{(2m_j)}(\xi_j)}\leq c2^{2m_j}B_{2m_j}\zeta^{2m_j}\lambda_r^\mu x_j^{\mu-2m_j}$.
Combining these two inequalities yields two positive constants $c$ and $\zeta>1$ such that for all  $\kappa>1$, $s\in\mathbb{N}$ with $s>1$ and $j\in\mathbb{Z}_{s-1}^{+}$
\begin{eqnarray*}
\mathcal{E}_{j}[\varphi_\kappa]
\leq c\lambda_r^\mu\dfrac{B_{2m_j}\zeta^{2m_j}h_j^{2m_j+1}x_j^{\mu-2m_j}}{(2m_j+1)!}
\leq c\dfrac{B_{2m_j+1}\nu^{(1+\mu)(j-1)+(2m_j+1)}}{(2m_j+1)!}\lambda_r^\mu\gamma^{(1+\mu)(s-1)}(1-\gamma)^{2m_j+1}.
\end{eqnarray*}
From Lemma \ref{lem:Sec5OrderSE2}  we conclude  that there
exists a positive constant $c$ such that for all  $\kappa>1$, $s\in\mathbb{N}$ with $s>1$ and
$j\in\mathbb{Z}_{s-1}^{+}$,
$\mathcal{E}_{j}[\varphi_\kappa]\leq c\lambda_r^\mu\gamma^{(1+\mu)(s-1)}(1-\gamma)^{2m_j+1}$.
Summing up the bound of errors $\mathcal{E}_{j}[\varphi_\kappa]$  over $j\in\mathbb{Z}_{s-1}$, we conclude that
there exists a positive constant $c$ such that for all  $\kappa>1$ and $s\in\mathbb{N}$
with $s>1$,
$\mathcal{E}_{\gamma}^{s}[f,g]\leq c\lambda_r^{1+\mu}\gamma^{(1+\mu)(s-1)}\Big{(}1+\sum_{j\in\mathbb{Z}_{s-1}^{+}}(1-\gamma)^{2m_j+1}\Big{)}$.
According to $m_j\neq m_l$ for $j,l\in\mathbb{Z}_{s-1}^+$ with $j\neq l$ proved in Lemma \ref{lem:Sec5OrderSE1}, we observe that $\set{2m_j+1: j\in\mathbb{Z}_{s-1}^+}\cup\set{0}\subset \mathbb{Z}:=\set{0,1,2,\ldots}$.
It follows that $1+\sum_{j\in\mathbb{Z}_{s-1}^{+}}(1-\gamma)^{2m_j+1}<\sum_{j\in\mathbb{Z}}(1-\gamma)^{j}=\gamma^{-1}$.
Substituting this result and the definition \eqref{Sec4:kappa} of $\lambda_r$  into the inequality above
yields the estimate \eqref{Sec5:ErrorSE}.

The bound of the number of functional evaluations used in quadrature \eqref{alm:Sec5SE} may be obtained directly.

When $g(x)=x$, we substitute $r=0$ and  $\kappa_{\sigma(0)}=\kappa$ into estimate \eqref{Sec5:ErrorSE} to yield the special result.
\end{proof}

We next develop a quadrature rule for the oscillatory integrals \eqref{Sec4:SOI}.
This is done by choosing variable numbers of quadrature nodes in the subintervals of $\Lambda$.
Specifically, for  $n\in\mathbb{N}$ and for the partition $\Pi_{\kappa,\sigma(r)}$ of $\Lambda$
with nodes defined  by \eqref{Sec4:PartitionOI}, we let
$m_j:=n+\lceil(n+1-j)(1-\alpha)\rceil$ and $N_j:=\left\lceil q_j\right\rceil$ for $j\in\mathbb{Z}_{n}^+$.
For each $j\in\mathbb{Z}^{+}_{n}$, we use $\mathcal{Q}_{N_j,\kappa,m_j}^{[x_{j-1},x_{j}]}[f,g]$ to approximate the integral $\mathcal{I}_\kappa^{[x_{j-1},x_{j}]}[f,g]$.
Integral $\mathcal{I}_\kappa^\Lambda[f,g]$ defined by  \eqref{Sec4:SOI}
is then approximated by the quadrature formula
\begin{equation}\label{alm:Sec5E}
\mathcal{Q}_{\kappa,n}^{\Lambda}[f,g]:=\sum_{j\in\mathbb{Z}_n^+}\mathcal{Q}_{N_j,\kappa,m_j}^{[x_{j-1},x_{j}]}[f,g].
\end{equation}
The next lemma provide an estimate of the errors
$\mathcal{E}_j[f,g]:=\abs{\mathcal{I}_{\kappa}^{[x_{j-1},x_{j}]}[f,g]
-\mathcal{Q}_{N_j,\kappa,m_j}^{[x_{j-1},x_{j}]}[f,g]}$
for $j\in\mathbb{Z}_{n}^+$.

\begin{lemma}\label{lem:Sec5SubE}
If $\Theta$ is of ${\rm Type}(\alpha,\infty,\set{0})$ for some  $\alpha\in (-1,1)$ and $g$ satisfies Assumption $\ref{Sec4:Assume}$, then there exists a positive  constant $c$ such that for all $\kappa>1$, $n\in\mathbb{N}$ and $j\in\mathbb{Z}_n^+$
\begin{eqnarray*}
\mathcal{E}_j[f,g]\leq c
\dfrac{\abs{g(1)}^{-\alpha}(\delta(r))^{\alpha-1}}{(m_j-1)!\kappa^2N_j^{m_j-1}}q_j^{m_j}x_{j-1}^{(\alpha-1)(r+1)}\left(\beta_n(r)\right)^{m_j}.
\end{eqnarray*}
\end{lemma}

This lemma may be proved in the same way as Lemma \ref{lem:Sec4SubE} with $m=m_j$.
We next estimate the error
$\mathcal{E}_{\kappa,n}^{\Lambda}[f,g]:=\abs{\mathcal{I}_\kappa^{\Lambda}[f,g]
-\mathcal{Q}_{\kappa,n}^{\Lambda}[f,g]}$.
To this end, we  establish two technical lemmas.

\begin{lemma}\label{lem:Sec5OrderP}
For all $n\in\mathbb{N}$ with $n>1$,
$\sum_{j\in\mathbb{Z}_{n}^+}1/(m_j-1)\leq 1$.
\end{lemma}
\begin{proof}
The proof is similar to that of Lemma \ref{lem:Sec3OrderP}.
For $n\in\mathbb{N}$ with $n>1$,
we need to estimate the lower bound of the set
$\set{m_j-1: j\in\mathbb{Z}_{n}^+}$.
By $\alpha<1$,  we have that $ m_j-1\geq n$ for $j\in\mathbb{Z}_{n}^+$.
It follows that
$\sum_{j\in\mathbb{Z}_{n}^+}1/(m_j-1)\leq\sum_{j\in\mathbb{Z}_{n}^+}1/n= 1$.
\end{proof}

For simple notation, we let $\tau_n(r):=\kappa_{\sigma(r)}^{1/n}\left(\kappa_{\sigma(r)}^{1/n}-1\right)$.

\begin{lemma}\label{lem:Sec5Pre}
There exists a positive constant $c$  such that for all $\kappa>1$, $j\in\mathbb{Z}_n^+$ and $n\in\mathbb{N}$
satisfying
\begin{equation}\label{Sec5:C}
\tau_n(r)\leq (n-1)/{\rm e},
\end{equation}
there holds $\left(\tau_n(r)\right)^{m_j-n}/(m_j-2)!\leq c(n-1)^{-{1}/{2}}$. Moreover, $\kappa_{\sigma(r)}^{1/n}\leq 2(n-1)^{1/2}$.
\end{lemma}
\begin{proof}
By $\alpha<1$  we see that $m_j-2\geq n-1$ for $j\in\mathbb{Z}_{n}^+$.
Condition \eqref{Sec5:C} implies
$n\geq 2$ since $ \tau_1(r)>0$ for  $\kappa>1$.
It follows  for  $n\in\mathbb{N}$ satisfying \eqref{Sec5:C} with $\tau_n(r)\leq1$ that
$\left(\tau_n(r)\right)^{m_j-n}/(m_j-2)!\leq 1/(m_j-2)!\leq (n-1)^{-{1}/{2}}$.
On the other hand,  for  $n\in\mathbb{N}$ satisfying \eqref{Sec5:C} with $\tau_n(r)>1$, we have
$\left(\tau_n(r)\right)^{m_j-n}/(m_j-2)!\leq \left(\tau_n(r)\right)^{m_j-2}/(m_j-2)!$.
According to inequality \eqref{Sec3:Stirling-coro} with $n:=m_j-2$, there exists a positive constant $c$ such that for all $n\in\mathbb{N}$ with $n>1$ and $j\in\mathbb{Z}_n^+$,
$1/(m_j-2)!\leq c (m_j-2)^{-{1}/{2}}\left({\rm e}/(m_j-2)\right)^{m_j-2}
 \leq c (n-1)^{-{1}/{2}}\left({\rm e}/(n-1)\right)^{m_j-2}$.
This together with the  inequality above ensures that
there exists a positive constant $c$ such that for all $\kappa>1$,  $j\in\mathbb{Z}_n^+$ and $n\in\mathbb{N}$ satisfying  \eqref{Sec5:C} with $\tau_n(r)>1$,
$\left(\tau_n(r)\right)^{m_j-n}/(m_j-2)!
\leq c (n-1)^{-{1}/{2}}\left({\rm e}\tau_n(r)/(n-1)\right)^{m_j-2}
\leq c (n-1)^{-{1}/{2}}$. This concludes the desired result.

We next prove the second inequality.
From condition \eqref{Sec5:C} we observe  that
$\left(\kappa_{\sigma(r)}^{1/n}-1\right)^2\leq\tau_n(r)\leq (n-1)/{\rm e}$.
Using the inequality
$\left((n-1)/{\rm e}\right)^{1/2}+1\leq 2(n-1)^{1/2}$ for $n>1$ yields that for $n\in\mathbb{N}$ satisfying \eqref{Sec5:C} with $\kappa_{\sigma(r)}^{1/n}>2$,
$\kappa_{\sigma(r)}^{{1}/{n}}\leq \left((n-1)/{\rm e}\right)^{1/2}+1\leq 2(n-1)^{1/2}$. Moreover, it follows for $n\in\mathbb{N}$ satisfying \eqref{Sec5:C} with $\kappa_{\sigma(r)}^{1/n}\leq2$,
$\kappa_{\sigma(r)}^{{1}/{n}}\leq 2(n-1)^{1/2}$.
Combining these two results we obtain the second conclusion.
\end{proof}

We are now ready to estimate  the error $\mathcal{E}_{\kappa,n}^{\Lambda}[f,g]$ and $\mathcal{N}\left(\mathcal{Q}_{\kappa,n}^{\Lambda}[f,g]\right)$.
Let $\rho_n(r):=1-\lambda_r^{1/n}$, where $\lambda_r$ is defined by \eqref{Sec4:kappa}.
We first consider the case $r=0$. For simple notation, let $\tau_n^0:=\tau_n(0)$ and $\rho_n^0:=\rho_n(0)$.

\begin{theorem}\label{thm:Sec5E0}
If $\Theta$ is of ${\rm Type}(\alpha,\infty,\set{0})$ for some  $\alpha\in (-1,1)$ and $g$ satisfies Assumption $\ref{Sec4:Assume}$ with $r=0$, then there exists a positive constant $c$  such that for all $\kappa>1$ and $n\in\mathbb{N}$ satisfying \eqref{Sec5:C},
\begin{equation}\label{Sec5:ErrorE0}
\mathcal{E}_{\kappa,n}^{\Lambda}[f,g]\leq c \abs{g(1)}^{-\alpha} \sigma_0\delta_0^{\alpha-2}(n-1)^{-1/2}\kappa^{-2}\kappa_{\sigma_0}\left(\rho_n^0\right)^n.
\end{equation}
For $n\in\mathbb{N}$, there holds the estimate
$\mathcal{N}\left(\mathcal{Q}_{\kappa,n}^{\Lambda}[f,g]\right)
\leq \left\lceil\sigma_0/\delta_0\right\rceil\big{(}2n^2+\lceil1-\alpha\rceil(n^2+n)\big{)}/2+1$.
\end{theorem}
\begin{proof}
We prove \eqref{Sec5:ErrorE0} by estimating $\mathcal{E}_{j}[f,g]$ and then summing them over $j\in\mathbb{Z}_{n}^+$. By employing Lemmas \ref{lem:Sec4Bound} and \ref{lem:Sec5SubE}  with $r=0$, and the choice of $N_j$, there exists a positive constant $c$ such that for all $\kappa>1$, $j\in\mathbb{Z}_{n}^+$ and
$n\in\mathbb{N}$,
$\mathcal{E}_j[f,g]\leq c
\dfrac{\abs{g(1)}^{-\alpha}\sigma_0\delta_0^{\alpha-2}}{(m_j-1)!\kappa^2}{x_{j-1}^{\alpha-1}}\left(\beta_n^0\right)^{m_j}$.
Applying $\beta_n^0=\lambda_0^{-{1}/{n}}\rho_n^0$ with
the definition of  $x_{j-1}$, we observe that
$x_{j-1}^{\alpha-1}\left(\beta_n^0\right)^{m_j}
\leq\left(\tau_n^0\right)^{m_j-n}\lambda_0^{-1}\left(\rho_n^0\right)^n$.
Combining these two inequalities yields a positive constant $c$  such that for all $\kappa>1$, $j\in\mathbb{Z}_n^+$ and $n\in\mathbb{N}$,
$\mathcal{E}_j[f,g]\leq c \dfrac{\abs{g(1)}^{-\alpha}\sigma_0\delta_0^{\alpha-2}}{(m_j-1)\kappa^2}
 \dfrac{\left(\tau_n^0\right)^{m_j-n}}{(m_j-2)!}\lambda_0^{-1}\left(\rho_n^0\right)^n$.
This together with the first result of Lemma \ref{lem:Sec5Pre} with $r=0$ ensures that there
exists a positive constant $c$  such that for all $\kappa>1$, $j\in\mathbb{Z}_n^+$ and
$n\in\mathbb{N}$ that satisfy \eqref{Sec5:C},
$\mathcal{E}_j[f,g]\leq c
(m_j-1)^{-1}\abs{g(1)}^{-\alpha}\sigma_0\delta_0^{\alpha-2}(n-1)^{-1/2}\kappa^{-2}\lambda_0^{-1}\left(\rho_n^0\right)^n$.
Summing up the inequality above over $j\in\mathbb{Z}_n^+$, we obtain that there exists a positive constant
$c$  such that for all $\kappa>1$ and $n\in\mathbb{N}$ satisfying \eqref{Sec5:C},
\begin{eqnarray*}
\mathcal{E}_{\kappa,n}^{\Lambda}[f,g]\leq c \abs{g(1)}^{-\alpha} \sigma_0\delta_0^{\alpha-2}(n-1)^{-1/2}\kappa^{-2}\lambda_0^{-1}\left(\rho_n^0\right)^n\sum_{j\in\mathbb{Z}_{n}^+}(m_j-1)^{-1}.
\end{eqnarray*}
This together with Lemma \ref{lem:Sec5OrderP} and $\lambda_0^{-1}=\kappa_{\sigma_0}$ leads to the estimate \eqref{Sec5:ErrorE0}.

According to formula \eqref{alm:Sec5E} and Lemma \ref{lem:Sec4Bound} with $r=0$, we  obtain that for $n\in\mathbb{N}$,
$\mathcal{N}\left(\mathcal{Q}_{\kappa,n}^{\Lambda}[f,g]\right)
 \leq\sum_{j\in\mathbb{Z}_n^+}\Big{\{}\left\lceil\sigma_0/\delta_0\right\rceil m_j+1\Big{\}}-(n-1)
 \leq\left\lceil\sigma_0/\delta_0\right\rceil\big{(}2n^2+\lceil1-\alpha\rceil(n^2+n)\big{)}/2+1$.
\end{proof}

Following Theorem \ref{thm:Sec5E0}, we have a special result for the case $g(x)=x$ for $x\in I$.

\begin{corollary}\label{cor:Sec5LinearE}
If $f$ is of ${\rm Type}(\mu,\infty,\set{0})$ for some $\mu\in(-1,1)$, then there exists a positive constant $c$  such that for all $\kappa>1$ and $n\in\mathbb{N}$ that satisfy \eqref{Sec5:C} with $r=0$,
$\mathcal{E}_{\kappa,n}^{\Lambda}[f,g]\leq c {(n-1)^{-{1}/{2}}\kappa^{-1}}\epsilon_n^{n}$,
where $\epsilon_n:=1-\kappa^{-1/n}$. For $n\in\mathbb{N}$, there holds the estimate
$\mathcal{N}\left(\mathcal{Q}_{\kappa,n}^{\Lambda}[f,g]\right)
 \leq \big{(}2n^2+\lceil1-\mu\rceil(n^2+n)\big{)}/2+1$.
\end{corollary}

Next, we consider the case $r>0$.

\begin{theorem}\label{thm:Sec5E}
If $\Theta$ is of ${\rm Type}(\alpha,\infty,\set{0})$ for some  $\alpha\in (-1,1)$ and $g$ satisfies Assumption $\ref{Sec4:Assume}$ with $r>0$, then there exists a positive constant $c$  such that for all $\kappa>1$ and $n\in\mathbb{N}$ satisfying \eqref{Sec5:C},
\begin{equation}\label{Sec5:ErrorE}
\mathcal{E}_{\kappa,n}^{\Lambda}[f,g]\leq c (r+1)\abs{g(1)}^{-\alpha} \sigma(r)(\delta(r))^{\alpha-2}\kappa^{-2}\kappa_{\sigma(r)}^{{1}/{(r+1)}}(\rho_n(r))^n.
\end{equation}
For $n\in\mathbb{N}$, there holds
$\mathcal{N}\left(\mathcal{Q}_{\kappa,n}^{\Lambda}[f,g]\right)\leq \left\lceil (r+1)\kappa_{\sigma(r)}^{r/(n(r+1))}\sigma(r)/\delta(r)\right\rceil\big{(}2n^2+\lceil1-\alpha\rceil(n^2+n)\big{)}/2+1$.
\end{theorem}
\begin{proof}
The proof is handled in the same way as that of Theorem \ref{thm:Sec5E0}.
Using $\beta_n(r)=\lambda_r^{-1/n}\rho_n(r)$ and the definition of $x_{j-1}$, we obtain that
$x_{j-1}^{(\alpha-1)(r+1)}(\beta_n(r))^{m_j}\leq\left(\tau_n(r)\right)^{m_j-n}\lambda_r^{-1}(\rho_n(r))^n$,
for $j\in\mathbb{Z}_n^+$.
By employing Lemmas \ref{lem:Sec4Bound} and \ref{lem:Sec5SubE}, the choice of $N_j$ and the inequality above, there exists a positive constant $c$  such that for all $\kappa>1$, $j\in\mathbb{Z}_n^+$ and $n\in\mathbb{N}$,
\begin{eqnarray*}
\mathcal{E}_j[f,g]\leq c \dfrac{(r+1)\abs{g(1)}^{-\alpha}\sigma(r)(\delta(r))^{\alpha-2}\kappa_{\sigma(r)}^{1/n}}{(m_j-1)\kappa^2}
\dfrac{\left(\tau_n(r)\right)^{m_j-n}}{(m_j-2)!}\lambda_r^{-1}(\rho_n(r))^n.
\end{eqnarray*}
Using the first result of Lemma \ref{lem:Sec5Pre}, we obtain that
there exists a positive constant $c$  such that for all $\kappa>1$, $j\in\mathbb{Z}_n^+$ and $n\in\mathbb{N}$  satisfying \eqref{Sec5:C},
\begin{eqnarray*}
\mathcal{E}_j[f,g]\leq c (r+1)\abs{g(1)}^{-\alpha}\sigma(r)(\delta(r))^{\alpha-2}\kappa_{\sigma(r)}^{1/n}(n-1)^{-1/2}\kappa^{-2}
\lambda_r^{-1}(\rho_n(r))^n/(m_j-1).
\end{eqnarray*}
This together with the second result of Lemma \ref{lem:Sec5Pre} and $\lambda_r^{-1}=\kappa_{\sigma(r)}^{{1}/{(r+1)}}$ ensures that there exists a positive constant $c$  such that for all $\kappa>1$,  $j\in\mathbb{Z}_n^+$ and
$n\in\mathbb{N}$ satisfying \eqref{Sec5:C},
\begin{equation*}
\mathcal{E}_j[f,g]\leq c
(r+1)\abs{g(1)}^{-\alpha}\sigma(r)(\delta(r))^{\alpha-2}\kappa^{-2}\kappa_{\sigma(r)}^{1/(r+1)}(\rho_n(r))^n/(m_j-1).
\end{equation*}
Summing up the inequality above over  $j\in\mathbb{Z}_n^+$ and
applying Lemma \ref{lem:Sec5OrderP}, we obtain the  estimate \eqref{Sec5:ErrorE}.

It remains to estimate $\mathcal{N}\left(\mathcal{Q}_{\kappa,n}^{\Lambda}[f,g]\right)$.
According to formula \eqref{alm:Sec5E} and Lemma \ref{lem:Sec4Bound}, we have for $n\in\mathbb{N}$
that
$\mathcal{N}\left(\mathcal{Q}_{\kappa,n}^{\Lambda}[f,g]\right)
\leq\sum_{j\in\mathbb{Z}_n^+}\Big{\{}\left\lceil q_j\right\rceil m_j+1\Big{\}}-(n-1)
\leq\left\lceil(r+1)\kappa_{\sigma(r)}^{r/(n(r+1))}\sigma(r)/\delta(r) \right\rceil\big{(}2n^2+\lceil1-\alpha\rceil(n^2+n)\big{)}/2+1$.
This yields the last conclusion.
\end{proof}

Note that the decay of the  bound of $\mathcal{E}_{\kappa,n}^{\Lambda}[f,g]$ given in Theorems \ref{thm:Sec5E0} and \ref{thm:Sec5E} is faster than the exponential decay with the base $\rho_{n_0}(r)$, where $n_0:=\min\set{n: n\in\mathbb{N}, n~\text{satisfying condition}~\eqref{Sec5:C}}$.
For $\kappa>1$,
we conclude that $0<\rho_{n_0}(r)<1$ and
$\lim\limits_{n\to\infty}(\rho_n(r))^n/(\rho_{n_0}(r))^n=0$.

The oscillatory integral \eqref{Sec2:OI0} may be computed by
$\mathcal{Q}_{\kappa,n}^{s,\gamma}[f,g]:=\mathcal{Q}_{\gamma}^{s}[f,g]+\mathcal{Q}_{\kappa,n}^{\Lambda}[f,g]$.
We next estimate the error $\mathcal{E}_{\kappa,n}^{\gamma,s}[f,g]$ of the quadrature formula $\mathcal{Q}_{\kappa,n}^{s,\gamma}[f,g]$ under reasonable hypotheses. Note that $-n\ln{\left(1-\gamma^{1+\mu}\right)}\geq \ln{\kappa_{\sigma(r)}}$ for $n\in\mathbb{N}$ implies that
$\rho_n(r)\leq \gamma^{1+\mu}$. If $r>0$ and $n=s$, then there exists a positive constant $c$ such that for all $\kappa>1$ and $n\in\mathbb{N}$ that satisfy \eqref{Sec5:C} and $-n\ln{\left(1-\gamma^{1+\mu}\right)}\geq \ln{\kappa_{\sigma(r)}}$, $\mathcal{E}_{\kappa,n}^{\gamma,s}[f,g]\leq c \kappa^{-(1+\mu)/(r+1)} \gamma^{(1+\mu)(n-1)}$. If $r=0$ and $n=s$, we consider two cases. In the case $-1<\mu<0$, there exists a positive constant $c$ such that for all $\kappa>1$ and $n\in\mathbb{N}$ satisfying \eqref{Sec5:C} and
$-n\ln{\left(1-\gamma^{1+\mu}\right)}\geq \ln{\kappa_{\sigma_0}}$, $\mathcal{E}_{\kappa,n}^{\gamma,s}[f,g] \leq c \kappa^{-1-\mu}\gamma^{(1+\mu)(n-1)}$. In the case $0\leq\mu<1$, there exists a positive constant $c$  such that for all $\kappa>1$ and $n\in\mathbb{N}$ satisfying \eqref{Sec5:C} and
$-n\ln{\left(1-\gamma^{1+\mu}\right)}\geq \ln{\kappa_{\sigma_0}}$,
$\mathcal{E}_{\kappa,n}^{\gamma,s}[f,g]\leq c \kappa^{-1}\gamma^{(1+\mu)(n-1)}$.


\section{Numerical Experiments}

In this section, we present numerical results of seven computational experiments to verify the approximation accuracy and computational efficiency of the proposed composite moment-free Filon-type (CMF) quadrature formulas, among which the quadrature formulas with a polynomial order of convergence will be denoted by CMFP and the quadrature formulas with an exponential order of convergence will be denoted by CMFE. The numerical experiments to be presented are divided into two groups. The first three examples are about the CMFP formulas and the last four examples concern with the CMFE formulas. We also compare the computational performance of the proposed quadrature formulas with that of the quadrature rules proposed in \cite{Graham1, Graham}.
The numerical results presented below are all obtained by using Matlab in a modest desktop (a Core 2 Quad with 4Gb of Ram memory).

In the first example, we test the CMFP formula \eqref{alm:Sec3P} for
computing the oscillator integral with smooth integrand and linear oscillator.

\begin{example}\label{egP:1} \rm
The purpose of this example is to confirm the theoretical estimate presented in Proposition \ref{prop:Sec3P} for the CMFP formula \eqref{alm:Sec3P} which approximates the integral $\mathcal{I}_\kappa[f,g]$. In this example, we consider the functions $f(x):={\rm e}^x$ and $g(x):=x$ for $x\in I$. The exact value of the corresponding integral can be computed exactly, and thus, the errors presented below are computed using the exact value.
\vspace{-1em}
\begin{center}
\footnotesize
\makeatletter
\def\@captype{table}
\makeatother
{\small \caption{Numerical results of CMFP for $f(x):={\rm e}^x$ and $g(x):=x$}
\vspace{0.2cm}
\label{Sec6:Tab1EgP1}}{
\begin{tabular}{c|c|c|c|c|c|c|c|c}
\hline
\multirow{2}{*}{$ \kappa$}
&\multicolumn{2}{c|}{ $n=5$} &\multicolumn{2}{c|}{$n=15$}
&\multicolumn{2}{c|}{ $n=20$}&\multicolumn{2}{c}{ $n=30$}\\
\cline{2-9}
&RE& $\mathcal{N}$&RE& $\mathcal{N}$&RE& $\mathcal{N}$&RE& $\mathcal{N}$ \\
\hline
$10^2$&4.68e-7 &21&8.93e-8 &61&4.13e-8 &81&1.15e-8 &121\\
$10^3$&2.43e-7 &21&2.21e-9 &61&1.24e-8 &81&1.95e-9 &121\\
$10^4$&4.62e-8 &21&2.99e-9 &61&1.77e-9 &81&1.62e-10&121\\
$10^5$&5.72e-9 &21&6.38e-10&61&3.47e-10&81&1.65e-11&121\\
$10^6$&1.36e-10&21&2.62e-10&61&1.26e-10&81&2.67e-12&121\\
$10^7$&4.54e-11&21&1.51e-11&61&5.93e-12&81&2.79e-12&121\\
\hline
\end{tabular}}
\end{center}

\vspace{-1em}
\begin{center}
\makeatletter
\def\@captype{figure}
\makeatother
\includegraphics[width=5.0cm]{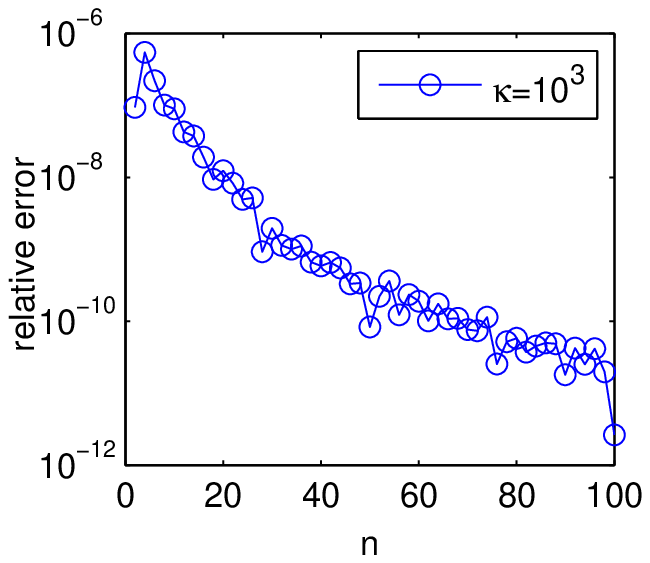}\ \hspace{2cm} \
\includegraphics[width=5.0cm]{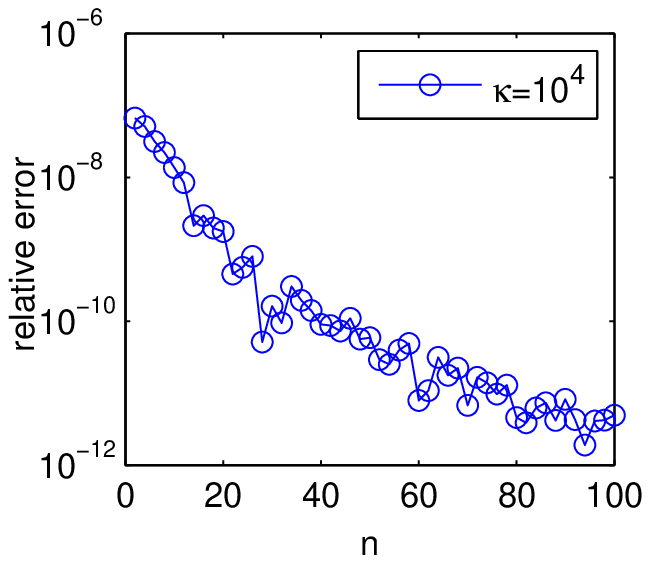}\\
\centerline{\small\ \ \ A:~RE of $\mathcal{Q}_{\kappa,n,4}[f,g]$\hspace{4cm}\ \ \ \ B:~RE of $\mathcal{Q}_{\kappa,n,4}[f,g]$}
\vspace{-1em}
{\small\caption{RE for $f(x):={\rm e}^x$ and $g(x):=x$}
\label{Sec6:Fig1EgP1}}
\end{center}

We present in Table \ref{Sec6:Tab1EgP1} the relative error (RE) of $\mathcal{Q}_{\kappa,n,4}[f,g]$ and the number $\mathcal{N}$ of functional evaluations used in the formula for different values of $\kappa$ and different choices of $n$. We plot in Figure \ref{Sec6:Fig1EgP1} the RE values of the quadrature formula for (A) $\kappa=10^3$ and for (B) $\kappa=10^4$. Figure \ref{Sec6:Fig1EgP1} and Table \ref{Sec6:Tab1EgP1} confirm that for a fixed $\kappa$ the approximation accuracy of the quadrature formula increases as $n$ grows and for a fixed $n$ it increases as $\kappa$ grows.


We verify in Figure \ref{Sec6:Fig2EgP1} the order in $1/\kappa$ of the error $\mathcal{E}_{\kappa,8,4}[f,g]$. We plot the  error $\mathcal{E}_{\kappa,8,4}[f,g]$ scaled by $\kappa^2$ in
(A) and that scaled by $\kappa^3$ in (B), for $n=8$ with $\kappa$ changing from $10^2$ to $10^3$.
Comparing (A) and (B) of Figure \ref{Sec6:Fig2EgP1}, we observe that the decay of the error $\mathcal{E}_{\kappa,8,4}[f,g]$  is faster than $\mathcal{O}(\kappa^{-2})$, but is slower than $\mathcal{O}(\kappa^{-4})$.

\begin{center}
\makeatletter
\def\@captype{figure}
\makeatother
\includegraphics[width=5.0cm]{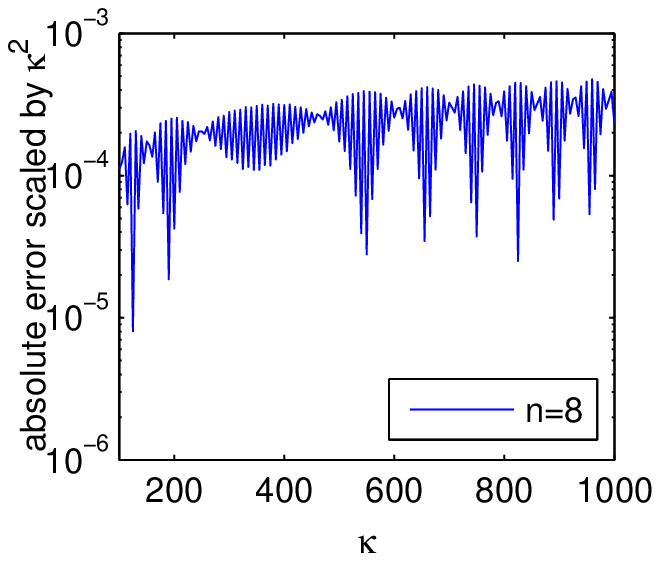}\ \hspace{2cm} \
\includegraphics[width=5.0cm]{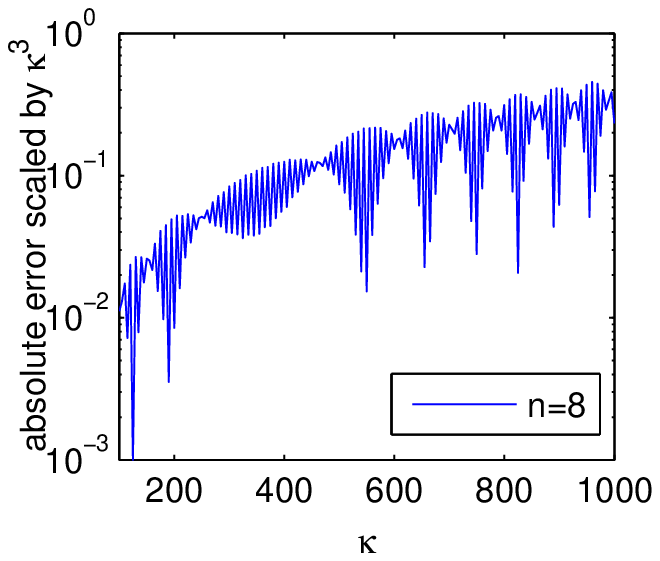}\\
\centerline{\small A:~$\kappa^2\mathcal{E}_{\kappa,8,4}[f,g]$\hspace{5cm} B:~$\kappa^3\mathcal{E}_{\kappa,8,4}[f,g]$}
\vspace{-1em}
{\small\caption{ Error $\mathcal{E}_{\kappa,8,4}[f,g]$ scaled by $\kappa^2$ and $\kappa^3$ for $f(x):={\rm e}^x$ and $g(x):=x$}
\label{Sec6:Fig2EgP1}}
\end{center}
\end{example}


We next compare the performance of the proposed  quadrature formulas with that of the existing
formulas described in \cite{Graham1, Graham}. To this end, we recall the quadrature formulas of
\cite{Graham1, Graham}. Following \cite{Graham1, Graham}, by $\beta\in(-1,1)$ we denote the index of singularity of $f$, $N\in\mathbb{N}$ denote the degree of the polynomial interpolant, $q>0$ denote the grading parameter and $M\in\mathbb{N}$ denote the number of the subintervals. The Filon-Clenshaw-Curtis (FCC) rule proposed in  \cite{Graham1} for the integral $\mathcal{I}_{\kappa}^{[-1,1]}[f,g]$ with a smooth integrand $f$ and a linear oscillator $g$ was designed by approximating $f$ by its Lagrange  interpolation $\mathcal{Q}_N(f)$ of degree $N$ at the Clenshaw-Curtis points $t_j:=\cos{(j\pi/N)}$ for $j\in\mathbb{Z}_N$. The FCC rule proposed in  \cite{Graham1} for the integral \eqref{Sec2:OI0} with a smooth integrand and a linear oscillator (that is, $g(x)=x$) is given by $\mathcal{I}_{\kappa,N}^{[0,1]}(f):={\rm e}^{{\rm i}\kappa/2}\mathcal{I}_{\kappa/2}^{[-1,1]}[\mathcal{Q}_N(\tilde{f}),g]/2$, where $\tilde{f}(x):=f((x+1)/2)$ for $x\in[-1,1]$. When $f$ has an integrable singularity at the origin, a composite Filon-Clenshaw-Curtis  (CFCC) quadrature rule was proposed in \cite{Graham}. Specifically, a positive integer $M$ is chosen, $q$ that satisfies the condition $q>(N+1-r)/(1+\beta-r)$ for some $0\leq r<1+\beta$ is picked, and a partition of the interval $I$ is formed with nodes $x_j:=M^{-q}j^q$ for $j\in\mathbb{Z}_M$. On all (but the first) subintervals of the partition, the FCC rule is applied. The treatment of the first subinterval depends on the value of the parameter $\beta$. For $\beta\in(-1,0]$, the integral $\mathcal{I}_{\kappa}^{[x_0,x_1]}[f,g]$ is approximated by zero. For $\beta\in(0,1)$ with $x_1\kappa\geq1$, the linear function interpolating the two pints $(x_0,f(x_0))$ and $(x_1,f(x_1))$ is used to approximate the function $f$ on $[x_0,x_1]$. While for $\beta\in(0,1)$ with $x_1\kappa<1$, the integrand of \eqref{Sec2:OI0} on $[x_0,x_1]$ is approximated by the linear function interpolating $(x_0,f(x_0){\rm e}^{{\rm i}\kappa g(x_0)})$ and $(x_1,f(x_1){\rm e}^{{\rm i}\kappa g(x_1)})$. The FCC rule with same number $N+1$ of points is used to calculate the integrals defined on the subintervals $[x_j, x_{j+1}]$ for $j\in\mathbb{Z}_{M-1}^+$. Let
$\widetilde{\mathcal{I}_{\kappa}}^{[x_0,x_1]}(f)$ denote the approximation of the integral $\mathcal{I}_{\kappa}^{[x_0,x_1]}[f,g]$. The CFCC rule is formed by
$\mathcal{I}_{\kappa,N,M,q}^{[0,1]}(f):=\widetilde{\mathcal{I}_{\kappa}}^{[x_0,x_1]}(f)+
\sum_{j\in\mathbb{Z}_{M-1}^+}h_j{\rm e}^{{\rm i}\kappa d_j}\mathcal{I}_{h_j\kappa}^{[-1,1]}[\mathcal{Q}_N(f_j),g]$,
where $h_j:=(x_{j+1}-x_j)/2$, $d_j:=(x_{j+1}+x_j)/2$ and $f_j(x):=f(h_jx+d_j)$ with $x\in[-1,1]$ for $j\in\mathbb{Z}_{M-1}^+$.
When $g$ is a strictly increasing nonlinear function and has a single stationary point at zero, we calculate the integral \eqref{Sec2:OI1} instead  by a change of variables $g(x)\to x$ and mapping $[g(0),g(1)]$ onto $I$.
The quadrature formulas proposed in \cite{Graham1, Graham} are implemented by using the Matlab code available in \cite{Dominguez}.

We now discuss the computation of the errors to be presented in Examples \ref{egP:2}, \ref{egP:3}, \ref{egE:2} and \ref{egE:3}. The true values of the integrals considered in these examples cannot be computed exactly. Therefore, we need to find their appropriate substitutions. This is done by the
means of the Gamma function. To this end, we let $\mathcal{I}_\kappa[0]:=\int_0^1\ln{x}{\rm e}^{{\rm i}\kappa x}{\rm d} x$ and $\mathcal{I}_\kappa[\alpha]:=\int_0^1x^{\alpha}{\rm e}^{{\rm i}\kappa x}{\rm d} x$,  for $\alpha\in(-1,0)\cup(0,1)$.
According to  \cite{Gradshteyn} (see also \cite{Graham}), we have that
$\mathcal{I}_\kappa[0]=-\left({\pi}/{2}+{\rm i}\varsigma+{\rm i}\ln{\kappa}\right)/\kappa+\mathcal{O}(1/\kappa^2)$ and
$\mathcal{I}_\kappa[\alpha]=\left(\Gamma(1+\alpha){\rm e}^{{\rm i}\pi(1+\alpha)}({\rm i}\kappa)^{-(1+\alpha)}+
{\rm e}^{{\rm i}\kappa}/({\rm i}\kappa)\right)(1+\mathcal{O}(1/\kappa))$, for $\alpha\in(-1,0)\cup(0,1)$,
where $\Gamma$ denotes the Gamma function, ${\rm arg}({\rm i}\kappa)=\pi/2$, and
$\varsigma$ denotes the Euler-Macheroti constant having an approximate value $0.5772$. We shall use $\mathcal{I}^a_\kappa[0]:=-\left({\pi}/{2}+{\rm i}\varsigma+{\rm i}\ln{\kappa}\right)/\kappa$ and
$\mathcal{I}^a_\kappa[\alpha]:=\Gamma(1+\alpha){\rm e}^{{\rm i}\pi(1+\alpha)}({\rm i}\kappa)^{-(1+\alpha)}+
{\rm e}^{{\rm i}\kappa}/({\rm i}\kappa)$,  for $\alpha\in(-1,0)\cup(0,1)$, respectively, to approximate the exact values of $\mathcal{I}_\kappa[0]$ and $\mathcal{I}_\kappa[\alpha]$, and substitute them when we compute errors of quadrature rules.

\newpage

\begin{center}
\footnotesize
\makeatletter
\def\@captype{table}
\makeatother
{\small  \caption{Comparison of CMFP and CFCC for $f_1(x):=x^{1/2}$ and $g(x):=x$}
\vspace{0.2cm}
\label{Sec6:Tab1EgP2}}{
\begin{tabular}{c|c|c|c|c|c|c|c|c}
\hline
 \multirow{3}{*}{$\kappa$}
&\multicolumn{4}{c|}{CMFP}&\multicolumn{4}{c}{CFCC}\\
\cline{2-9}
&\multicolumn{2}{c|}{$n=5$}&\multicolumn{2}{c|}{$n=10$}&\multicolumn{2}{c|}{$M=7$}
&\multicolumn{2}{c}{$M=14$}\\
\cline{2-9}
&RE&$\mathcal{N}$&RE&$\mathcal{N}$&RE&$\mathcal{N}$&RE&$\mathcal{N}$\\
\hline
$10^2$&5.03e-3&37&5.10e-3&77&5.09e-3&38&5.09e-3&80\\
$10^3$&5.80e-4&37&5.12e-4&77&5.65e-4&38&5.14e-4&80\\
$10^4$&8.40e-5&37&5.41e-5&77&1.52e-4&38&5.10e-5&80\\
$10^5$&8.58e-5&37&6.94e-6&77&7.65e-5&38&6.33e-6&80\\
$10^6$&3.26e-5&37&5.24e-6&77&2.17e-4&38&3.06e-6&80\\
$10^7$&3.02e-5&37&2.55e-6&77&3.84e-4&38&8.58e-6&80\\
\hline
\end{tabular}}
\end{center}

In the next two examples, we test the efficiency of the quadrature rule proposed in Section 4 for
calculating the oscillatory integrals  with singular $f$ and with/without a stationary point of $g$.

\begin{example}\label{egP:2} \rm
This example  is designed to verify the theoretical estimates given in Theorems \ref{thm:Sec4SP} and \ref{thm:Sec4P0} for the CMFP formula $\mathcal{Q}_{\kappa,n,m}^{s,\widetilde{m}}[f,g]$, which approximates the integral $\mathcal{I}_\kappa[f,g]$.
We consider three functions $f_1(x):=x^{1/2}$,  $f_2(x):=\ln{x}$, $f_3(x):=x^{-1/2}$ and $g(x):=x$ for $x\in I$.
When we compute the errors of the quadrature formulas, the true values of the integrals $\mathcal{I}_\kappa[f_j,g]$ are computed by using $\mathcal{I}^a_\kappa[\alpha_j]$,  for $j\in\mathbb{Z}_3^+$, respectively, where $\alpha_1:=1/2$, $\alpha_2:=0$ and $\alpha_3:=-1/2$.

\begin{center}
\footnotesize
\makeatletter
\def\@captype{table}
\makeatother
{\small \caption{Comparison of CMFP and CFCC for $f_2(x):=\ln{x}$ and $g(x):=x$}
\vspace{0.2cm}
\label{Sec6:Tab2EgP2}}{
\begin{tabular}{c|c|c|c|c|c|c|c|c}
\hline
 \multirow{3}{*}{$\kappa$}
 &\multicolumn{4}{c|}{CMFP}&\multicolumn{4}{c}{CFCC}\\
 \cline{2-9}
&\multicolumn{2}{c|}{$n=5$}&\multicolumn{2}{c|}{$n=10$}&\multicolumn{2}{c|}{$M=7$}
&\multicolumn{2}{c}{$M=14$}\\
\cline{2-9}
&RE&$\mathcal{N}$&RE&$\mathcal{N}$&RE&$\mathcal{N}$&RE&$\mathcal{N}$\\
\hline
$10^2$&1.75e-3&37&1.86e-3&77&4.15e-3&37&1.86e-3&79\\
$10^3$&1.72e-3&37&1.19e-4&77&8.87e-3&37&1.93e-4&79\\
$10^4$&2.90e-3&37&1.15e-4&77&1.59e-2&37&1.65e-4&79\\
$10^5$&6.23e-3&37&1.42e-4&77&5.14e-3&37&2.45e-4&79\\
$10^6$&1.08e-2&37&9.16e-4&77&4.76e-2&37&3.99e-4&79\\
$10^7$&1.60e-2&37&1.33e-3&77&1.53e-2&37&1.56e-3&79\\
\hline
\end{tabular}}
\end{center}

Numerical results of this example are presented in Tables \ref{Sec6:Tab1EgP2}, \ref{Sec6:Tab2EgP2} and \ref{Sec6:Tab3EgP2}.  We compare the RE values of the formula
$\mathcal{Q}_{\kappa,n,4}^{n,4}[f_1,g]$ and the CFCC formula $\mathcal{I}_{\kappa,6,M,8}^{[0,1]}(f_1)$ in Table \ref{Sec6:Tab1EgP2}, those of the formula $\mathcal{Q}_{\kappa,n,4}^{n,4}[f_2,g]$ and the CFCC formula $\mathcal{I}_{\kappa,6,M,12}^{[0,1]}(f_2)$ in Table \ref{Sec6:Tab2EgP2}, and those of the formula $\mathcal{Q}_{\kappa,n,4}^{n,4}[f_3,g]$ and  the CFCC formula $\mathcal{I}_{\kappa,6,M,16}^{[0,1]}(f_3)$ in Table \ref{Sec6:Tab3EgP2}. These numerical results show that overall the CMFP formula has higher approximation accuracy than the CFCC rule.

\begin{center}
\footnotesize
\makeatletter
\def\@captype{table}
\makeatother
{\small  \caption{Comparison of CMFP and CFCC for $f_3(x):=x^{-1/2}$ and $g(x):=x$}
\vspace{0.2cm}
\label{Sec6:Tab3EgP2}}{
\begin{tabular}{c|c|c|c|c|c|c|c|c}
\hline
 \multirow{3}{*}{$\kappa$}
 &\multicolumn{4}{c|}{CMFP}&\multicolumn{4}{c}{CFCC}\\
 \cline{2-9}
&\multicolumn{2}{c|}{$n=5$}&\multicolumn{2}{c|}{$n=10$}&\multicolumn{2}{c|}{$M=7$}
&\multicolumn{2}{c}{$M=14$}\\
\cline{2-9}
&RE&$\mathcal{N}$&RE&$\mathcal{N}$&RE&$\mathcal{N}$&RE&$\mathcal{N}$\\
\hline
$10^2$&2.89e-2&37&1.25e-3&77&1.18e-2&37&4.07e-4&79\\
$10^3$&2.50e-2&37&9.23e-4&77&1.01e-2&37&6.62e-4&79\\
$10^4$&3.14e-2&37&7.49e-4&77&2.57e-2&37&1.05e-3&79\\
$10^5$&3.38e-2&37&5.54e-4&77&4.49e-1&37&4.82e-3&79\\
$10^6$&8.55e-2&37&2.22e-3&77&1.12e-1&37&1.70e-2&79\\
$10^7$&1.12e-1&37&5.93e-3&77&2.06   &37&4.02e-2&79\\
\hline
\end{tabular}}
\end{center}

\end{example}

\begin{example}\label{egP:3} \rm
This example is to verify the estimates proved in Theorems \ref{thm:Sec4SP} and \ref{thm:Sec4P} for the CMFP formula $\mathcal{Q}_{\kappa,n,m}^{s,\widetilde{m}}[f,g]$, which approximates  $\mathcal{I}_\kappa[f,g]$.
In this example, we consider the functions $f(x):=x^{-1/2}$ and $g(x):=x^2$ for $x\in I$. When we compute the errors of the quadrature formulas, the true value of this integral is computed by  using $I^a_\kappa[-3/4]/2$.

\begin{center}
\makeatletter
\def\@captype{figure}
\makeatother
\includegraphics[width=5.0cm]{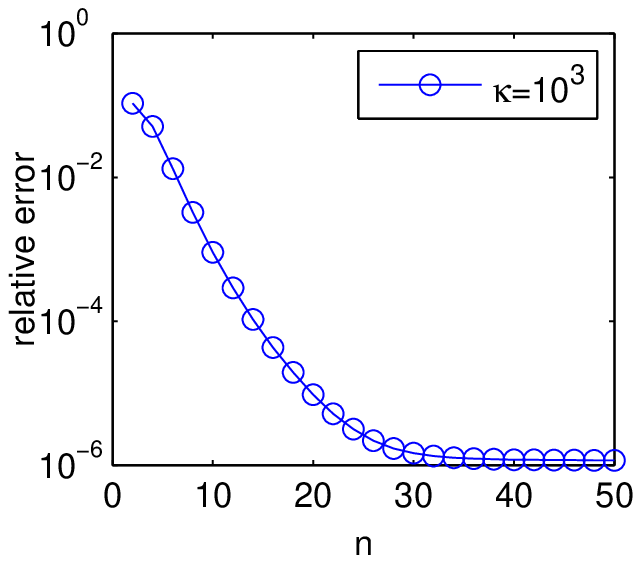}\ \hspace{2.5cm} \
\includegraphics[width=5.0cm]{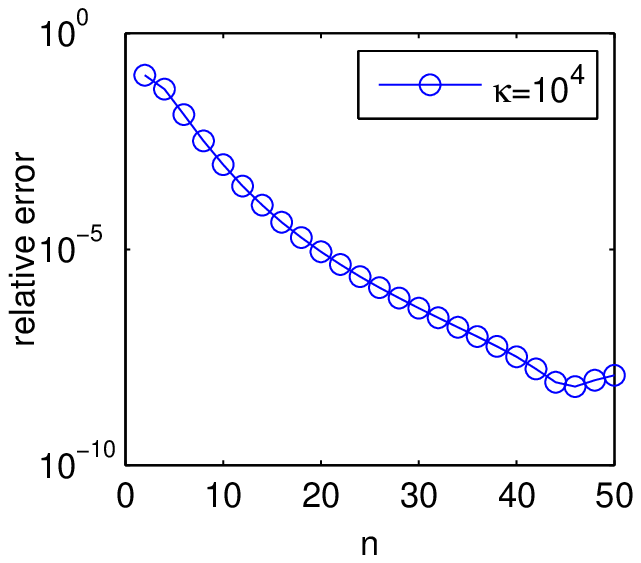}\\
\centerline{\small A:~RE of $\mathcal{Q}_{\kappa,n,4}^{n,4}[f,g]$  \hspace{4cm}\ \ \  B:~RE of $\mathcal{Q}_{\kappa,n,4}^{n,4}[f,g]$ }
\vspace{-1em}
{\small\caption{RE for  $f(x):=x^{-1/2}$ and $g(x):=x^2$
\label{Sec6:Fig1EgP3}}}
\end{center}

\begin{center}
\makeatletter
\def\@captype{figure}
\makeatother
\includegraphics[width=5.0cm]{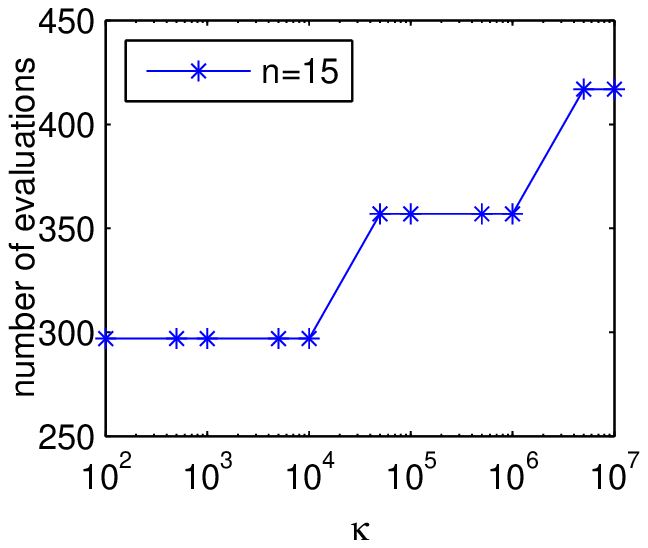}\ \hspace{2.5cm}  \
\includegraphics[width=5.0cm]{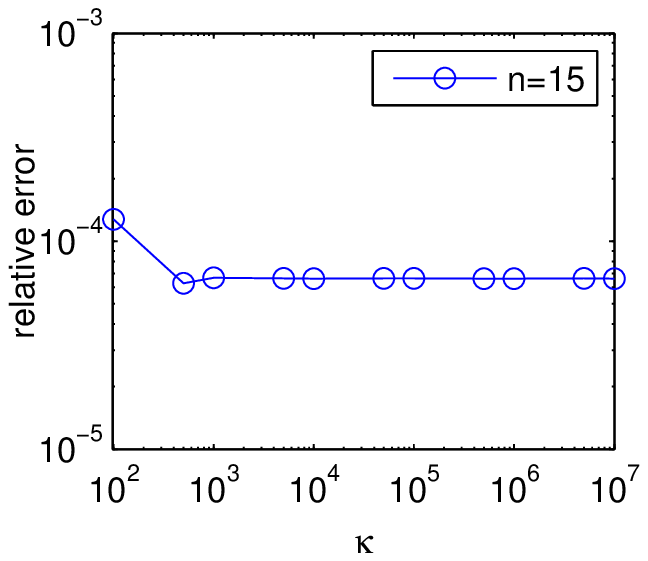}\\
\centerline{\small A:~$\mathcal{N}(\mathcal{Q}_{\kappa,n,4}^{n,4}[f,g])$ \hspace{5cm} B:~RE of $\mathcal{Q}_{\kappa,n,4}^{n,4}[f,g]$}
\vspace{-1em}
{\small\caption{ $\mathcal{N}$ and RE for $f(x):=x^{-1/2}$ and $g(x):=x^2$
\label{Sec6:Fig2EgP3}}}
\end{center}

Numerical results of this example are presented in Figure \ref{Sec6:Fig1EgP3}--\ref{Sec6:Fig4EgP3} and Table \ref{Sec6:Tab1EgP3}.
We plot in Figure \ref{Sec6:Fig1EgP3} the RE values of the formula  $\mathcal{Q}_{\kappa,n,4}^{n,4}[f,g]$ for (A) $\kappa=10^3$ and for (B) $\kappa=10^4$. We present in Figure \ref{Sec6:Fig2EgP3} the number $\mathcal{N}$ of functional evaluations used in the formula $\mathcal{Q}_{\kappa,15,4}^{15,4}[f,g]$ (A) and the RE values (B), and  those of the formula $\mathcal{Q}_{\kappa,30,4}^{30,4}[f,g]$ in Figure \ref{Sec6:Fig3EgP3}.
We plot in Figure \ref{Sec6:Fig4EgP3} the RE values obtained from  the CMFP and CFCC formulas  by using the same number of functional evaluations for (A) $\kappa=10^4$, for (B) $\kappa=10^5$ and for (C) $\kappa=10^6$. The RE values of the formula  $\mathcal{Q}_{\kappa,n,4}^{n,4}[f,g]$ and those of the CFCC formula  $\mathcal{I}_{\kappa,4,M,21}^{[0,1]}\left((f/g')\circ g^{(-1)}\right)$ are reported in  Table \ref{Sec6:Tab1EgP3}.

\begin{center}
\makeatletter
\def\@captype{figure}
\makeatother
\includegraphics[width=5.0cm]{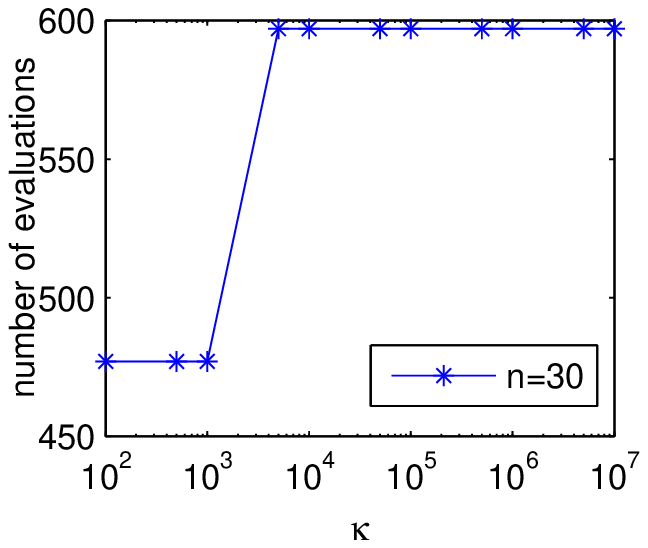}\ \hspace{2.5cm}  \
\includegraphics[width=5.0cm]{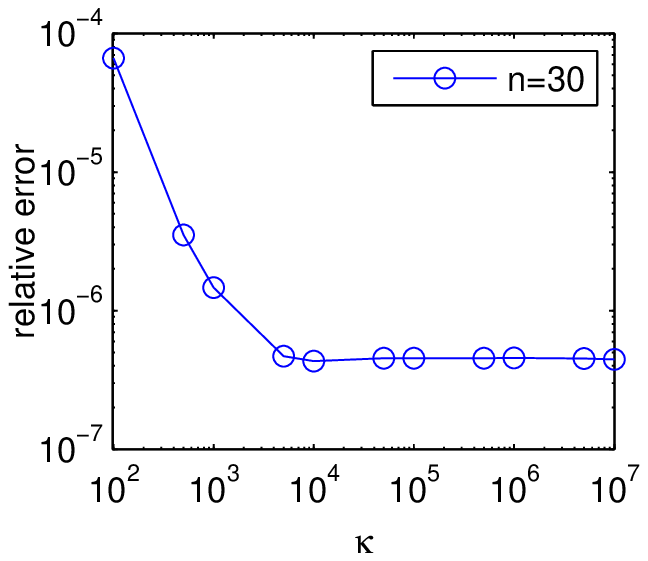}\\
\centerline{\small A:~$\mathcal{N}(\mathcal{Q}_{\kappa,n,4}^{n,4}[f,g])$ \hspace{5cm} B:~RE of $\mathcal{Q}_{\kappa,n,4}^{n,4}[f,g]$}
\vspace{-1em}
{\small\caption{ $\mathcal{N}$ and RE for $f(x):=x^{-1/2}$ and $g(x):=x^2$
\label{Sec6:Fig3EgP3}}}
\end{center}

From these numerical results, we have the following conclusions.
For a fixed  $\kappa$ the approximation accuracy of the CMFP formula increases as $n$ grows.
The CMFP formula achieves an accuracy of an order higher than the CFCC formula for large $\kappa$ and is effective and convenient for implementation.

\newpage

\begin{center}
\footnotesize
\makeatletter
\def\@captype{table}
\makeatother
{\small \caption{ Comparison of CMFP and CFCC for $f(x):=x^{-1/2}$ and $g(x):=x^2$}
\vspace{0.2cm}
\label{Sec6:Tab1EgP3}}{
\begin{tabular}{c|c|c|c|c|c|c|c|c}
\hline
 \multirow{3}{*}{$\kappa$}
&\multicolumn{4}{c|}{CMFP}&\multicolumn{4}{c}{CFCC}\\
\cline{2-9}
&\multicolumn{2}{c|}{ $n=30$ }&\multicolumn{2}{c|}{$n=50$}
&\multicolumn{2}{c|}{$M=200$}&\multicolumn{2}{c}{$M=400$}\\
\cline{2-9}
&RE&$\mathcal{N}$&RE&$\mathcal{N}$&RE&$\mathcal{N}$&RE&$\mathcal{N}$\\
\hline
$10^2$&6.62e-5&477&6.59e-5&797&6.60e-5&797&6.61e-5&1597\\
$10^3$&1.47e-6&477&1.17e-6&797&8.81e-7&797&9.50e-7&1597\\
$10^4$&4.33e-7&597&1.21e-8&797&6.01e-7&797&3.42e-7&1597\\
$10^5$&4.52e-7&597&8.95e-9&797&4.57e-7&797&3.92e-7&1597\\
$10^6$&4.55e-7&597&9.31e-9&997&2.91e-7&797&3.07e-7&1597\\
$10^7$&4.45e-7&597&8.88e-9&997&5.64e-7&797&3.72e-7&1597\\
\hline
\end{tabular}}
\end{center}

\begin{center}
\makeatletter
\def\@captype{figure}
\makeatother
\includegraphics[width=5.0cm]{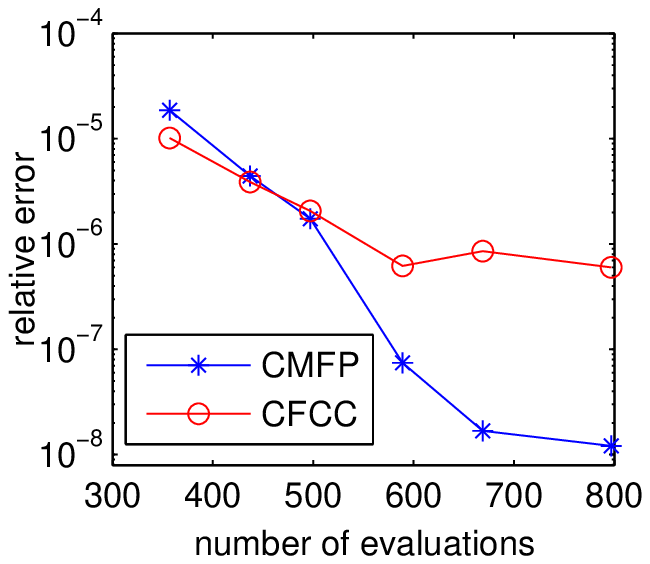}\hspace{0.8cm}
\includegraphics[width=5.0cm]{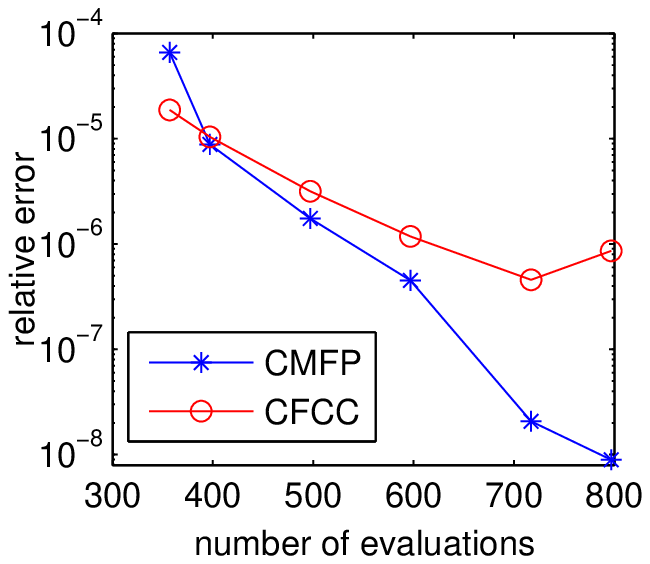}\hspace{0.8cm}
\includegraphics[width=5.0cm]{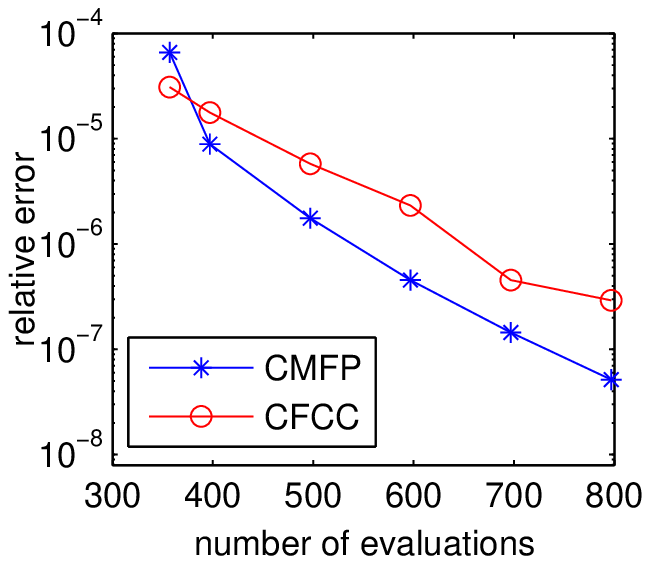}\\
\centerline{\small A:~ $\kappa=10^4$\hspace{4cm} B:~ $\kappa=10^5$\hspace{4cm} C:~ $\kappa=10^6$}
\vspace{-1em}
{\small\caption{ Comparison of CMFP and CFCC for $f(x):=x^{-1/2}$ and $g(x):=x^2$
\label{Sec6:Fig4EgP3}}}
\end{center}
\end{example}

The numerical results of these three examples indicate that the CMFP formulas  are stable, easy to implement and computationally inexpensive. For irregular oscillator, the performance of the CMFP formulas  is superior to that of the FCC and CFCC formulas. The advantage of the proposed formulas is especially obvious for the oscillator with stationary points. We next show the numerical results  calculated by the CMFE formulas.

In  Example \ref{egE:1},  we test the CMFE formula \eqref{alm:Sec3E} for
computing the oscillator integral considered in Example \ref{egP:1} and compare with the CMFP formula \eqref{alm:Sec3P}.

\begin{example}\label{egE:1} \rm
This example is to confirm the estimate shown in Theorem \ref{thm:Sec3E} for the CMFE formula \eqref{alm:Sec3E} which approximates the integral $\mathcal{I}_\kappa[f,g]$. We consider in this example the functions $f(x):={\rm e}^x$ and $g(x):=x$ for $x\in I$, the same as those in Example \ref{egP:1}.

\vspace{-1em}
\begin{center}
\footnotesize
\makeatletter
\def\@captype{table}
\makeatother
{\caption{Numerical results of CMFE for $f(x):={\rm e}^x$ and $g(x):=x$}
\vspace{0.2cm}
\label{Sec6:Tab1EgE1}}{
\begin{tabular}{c|c|c|c|c}
\hline
\multirow{2}{*}{$ \kappa$}
&\multicolumn{2}{c|}{ $n=4$ }&\multicolumn{2}{c}{$n=5$}\\
\cline{2-5}
&RE&$\mathcal{N}$&RE&$\mathcal{N}$\\
\hline
$10^2$&9.65e-14&26&2.24e-13&47\\
$10^3$&2.17e-15&26&6.39e-15&47\\
$10^4$&7.36e-17&26&1.62e-15&47\\
$10^5$&1.14e-17&26&5.70e-17&47\\
$10^6$&1.30e-16&26&2.61e-16&47\\
$10^7$&1.50e-16&26&2.32e-16&47\\
\hline
\end{tabular}}
\end{center}

Numerical results of this example are reported in Table \ref{Sec6:Tab1EgE1} and Figure  \ref{Sec6:Fig1EgE1}.
We list in Table \ref{Sec6:Tab1EgE1} the RE values of  the formula $\mathcal{Q}_{\kappa,n}[f,g]$.
We depict the RE values of the formula $\mathcal{Q}_{\kappa,n}[f,g]$ in Figure \ref{Sec6:Fig1EgE1} (A)
and  the  error $\mathcal{E}_{\kappa,n}[f,g]$ scaled by $\kappa^{n+1}$ in Figure \ref{Sec6:Fig1EgE1} (B), for $n=4$ with $\kappa$ changing from $100$ to $5000$. From Table \ref{Sec6:Tab1EgE1} and Figure \ref{Sec6:Fig1EgE1}, we observe that the approximation accuracy of  the formula $\mathcal{Q}_{\kappa,n}[f,g]$ increases as $\kappa$ grows for a fixed $n$ and the asymptotic order of convergence is $\mathcal{O}(\kappa^{-n-1})$ for $n=4$. It concurs with the theoretical estimate. Comparing Figure \ref{Sec6:Fig1EgE1} (B) with Figure \ref{Sec6:Fig2EgP1} in Example \ref{egP:1}, we can  obtain an order of convergence faster than $\mathcal{O}(\kappa^{-3})$ by  using variable number of quadrature nodes in the subintervals.

\begin{center}
\makeatletter
\def\@captype{figure}
\makeatother
\includegraphics[width=5.0cm]{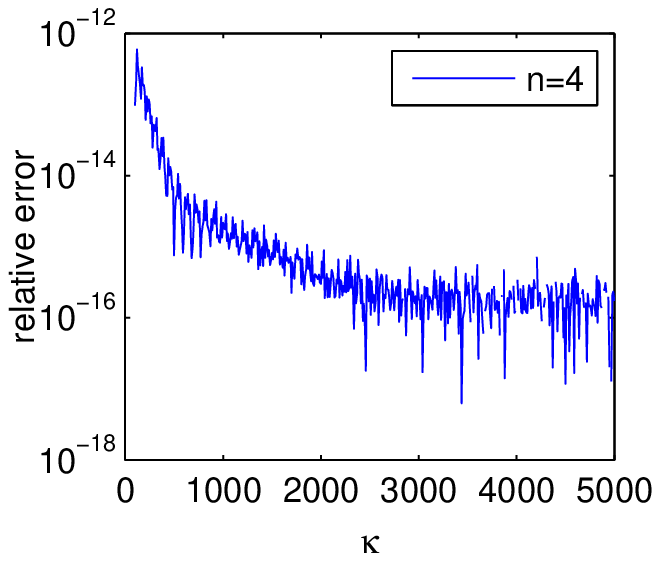}\ \hspace{2.5cm} \
\includegraphics[width=5.0cm]{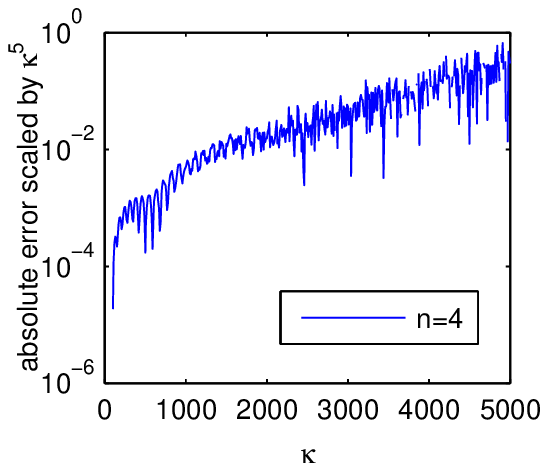}\\
\centerline{\small A:~RE of $\mathcal{Q}_{\kappa,n}[f,g]$\hspace{5cm} B:~$\kappa^{n+1}\mathcal{E}_{\kappa,n}[f,g]$}
\vspace{-1em}
\caption{\small RE  and error $\mathcal{E}_{\kappa,n}[f,g]$ scaled by $\kappa^{n+1}$ of Example \ref{egE:1}
\label{Sec6:Fig1EgE1}}
\end{center}
\end{example}

In the following two examples, we validate the efficiency of the quadrature rule proposed in Section 5 for
calculating the oscillatory integrals  with a stationary point, without/with singular $f$.

\begin{example}\label{egE:2}\rm
This example is to verify the estimates established in Theorems \ref{thm:Sec5SE} and \ref{thm:Sec5E} for the CMFE formula $\mathcal{Q}_{\kappa,n}^{s,\gamma}[f,g]$, which approximates the integral $\mathcal{I}_\kappa[f,g]$. We consider in this example the functions $f(x):=1$ and $g(x):=x^3$ for $x\in I$.
When we compute the errors of the quadrature formulas, the true value of this integral is computed by using $I^a_\kappa[-2/3]/3$.

\begin{center}
{\footnotesize
\makeatletter
\def\@captype{table}
\makeatother
{\small  \caption{ Comparison of CMFE and CFCC for $f(x):=1$ and $g(x):=x^3$}
\vspace{0.2cm}
\label{Sec6:Tab1EgE2}}{
\begin{tabular}{c|c|c|c|c|c|c|c|c}
\hline
 \multirow{3}{*}{$\kappa$}
&\multicolumn{4}{c|}{CMFE}&\multicolumn{4}{c}{CFCC}\\
\cline{2-9}
&\multicolumn{2}{c|}{$n=3$ }&\multicolumn{2}{c|}{$n=4$}
&\multicolumn{2}{c|}{$M=400$}&\multicolumn{2}{c}{$M=800$}\\
\cline{2-9}
&RE&$ \mathcal{N}$&RE&$\mathcal{N}$&RE&$\mathcal{N}$&RE&$\mathcal{N}$\\
\hline
$10^2$&1.17e-4 &307 &1.17e-4 &365 &1.18e-4 &1597 &1.17e-4&3197\\
$10^3$&2.49e-6 &407 &2.48e-6 &467 &2.62e-6 &1597 &2.43e-6&3197\\
$10^4$&5.36e-8 &607 &5.36e-8 &603 &5.07e-7 &1597 &3.04e-7&3197\\
$10^5$&1.08e-9 &907 &9.38e-10&841 &5.82e-7 &1597 &3.76e-7&3197\\
$10^6$&9.45e-11&1427&4.00e-10&1156&7.22e-7 &1597 &4.04e-7&3197\\
$10^7$&7.13e-11&2287&7.22e-10&1657&6.10e-7 &1597 &4.09e-7&3197\\
\hline
\end{tabular}}
}
\end{center}

Numerical results of this example are presented in  Table \ref{Sec6:Tab1EgE2} and Figure \ref{Sec6:Fig1EgE2}.
We list in Table \ref{Sec6:Tab1EgE2} the RE values of the formula $\mathcal{Q}_{\kappa,n}^{7,0.02}[f,g]$ and those of the CFCC formula  $\mathcal{I}^{[0,1]}_{\kappa,4,M,16}\left((f/g')\circ g^{(-1)}\right)$. In (A) of Figure \ref{Sec6:Fig1EgE2}, we depict the RE values of the formula $\mathcal{Q}_{\kappa,3}^{7,0.02}[f,g]$, which are listed in Row 2 of Table \ref{Sec6:Tab1EgE2} and the RE values obtained from the CFCC formula $\mathcal{I}^{[0,1]}_{\kappa,4,M,16}\left((f/g')\circ g^{(-1)}\right)$, where the numbers $\mathcal{N}=309$, $409$, $609$, $909$, $1429$, $2289$ of functional evaluations used in the formula are comparable to those used in  the CMFE formula recorded in Row 3 of Table \ref{Sec6:Tab1EgE2}. In (B) of Figure \ref{Sec6:Fig1EgE2}, we compare $\mathcal{N}(\mathcal{Q}_{\kappa,3}^{7,0.02}[f,g])$ with $(r+1)n^2\kappa^{1/n}$ with $r=2$ and $n=3$.

\begin{center}
\makeatletter
\def\@captype{figure}
\makeatother
\includegraphics[width=5.0cm]{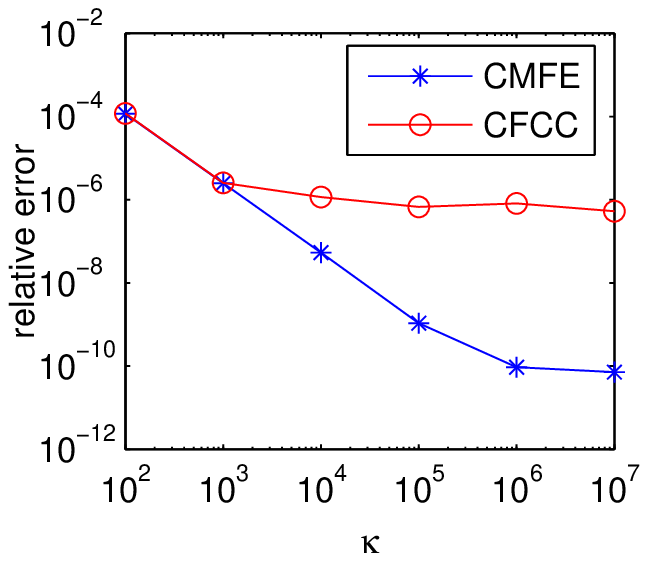}\ \hspace{2.5cm} \
\includegraphics[width=5.0cm]{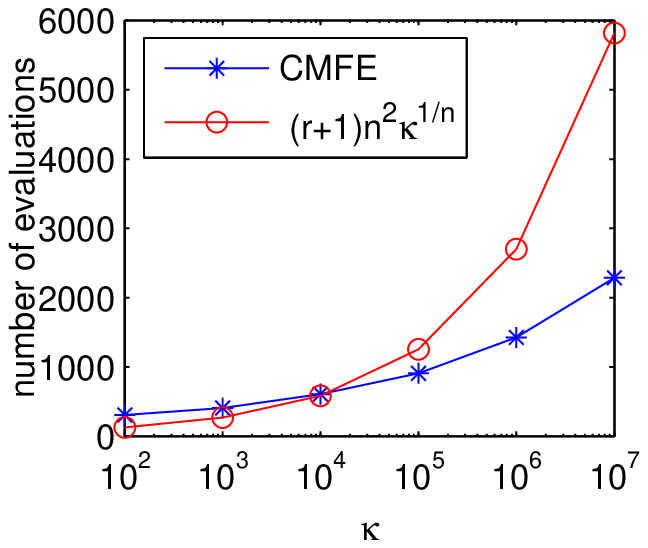}\\
\centerline{\small A:~RE of $\mathcal{Q}_{\kappa,3}^{7,0.02}[f,g]$ \hspace{4cm} B:~ $\mathcal{N}(\mathcal{Q}_{\kappa,3}^{7,0.02}[f,g])$}
\vspace{-1em}
{\small\caption{ RE and $\mathcal{N}$ for $f(x):=1$ and $g(x):=x^3$
\label{Sec6:Fig1EgE2}}}
\end{center}

\newpage
\begin{center}
\footnotesize
\makeatletter
\def\@captype{table}
\makeatother
{\small \caption{ Comparison of CMFE and CFCC for $f(x):=x^{-1/2}$ and $g(x):=x^2$}
\vspace{0.2cm}
\label{Sec6:Tab1EgE3}}{
\begin{tabular}{c|c|c|c|c|c|c|c|c}
\hline
\multirow{3}{*}{$\kappa$}
&\multicolumn{4}{c|}{CMFE}&\multicolumn{4}{c}{CFCC}\\
\cline{2-9}
&\multicolumn{2}{c|}{$n=3$}&\multicolumn{2}{c|}{$n=4$}
&\multicolumn{2}{c|}{$M=100$}&\multicolumn{2}{c}{$M=300$}\\
\cline{2-9}
&RE&$ \mathcal{N}$&RE&$\mathcal{N}$&RE&$\mathcal{N}$&RE&$\mathcal{N}$\\
\hline
$10^2$&6.59e-5&502 &6.59e-5&537&6.48e-5&397&6.58e-5&1197\\
$10^3$&1.19e-6&544 &1.16e-6&572&3.29e-6&397&1.47e-6&1197\\
$10^4$&6.23e-8&607 &1.75e-8&642&6.56e-6&397&4.52e-7&1197\\
$10^5$&4.50e-8&691 &2.74e-9&712&1.04e-5&397&3.28e-7&1197\\
$10^6$&1.90e-8&829 &2.13e-9&817&1.77e-5&397&4.93e-7&1197\\
$10^7$&1.09e-8&1027&4.12e-9&922&3.15e-5&397&4.45e-7&1197\\
\hline
\end{tabular}}
\end{center}

We conclude from the results in Table \ref{Sec6:Tab1EgE2} and Figure \ref{Sec6:Fig1EgE2} (A) that for a fixed $n$ the approximation accuracy of the formula $\mathcal{Q}_{\kappa,n}^{7,0.02}[f,g]$ increases as  $\kappa$ grows and $\mathcal{N}(\mathcal{Q}_{\kappa,n}^{7,0.02}[f,g])$ increases slow as $\kappa$ grows. Furthermore,
from the results presented in  Table \ref{Sec6:Tab1EgE2}, we observe that for $n=3$, $4$ and large $\kappa$, quadrature formula $\mathcal{N}(\mathcal{Q}_{\kappa,n}^{7,0.02}[f,g])$ has better approximation accuracy than the CFCC formula $\mathcal{I}^{[0,1]}_{\kappa,4,M,16}\left((f/g')\circ g^{(-1)}\right)$, which uses more number of functional evaluations. In fact, from Figure \ref{Sec6:Fig1EgE2} (B), we see that the number of functional evaluations used in formula $\mathcal{Q}_{\kappa,3}^{7,0.02}[f,g]$ is significantly less than $27\kappa^{1/3}$ for large $\kappa$.
\end{example}

\begin{example}\label{egE:3}\rm
The purpose of this example is to test the  estimates proved in Theorems \ref{thm:Sec5SE} and \ref{thm:Sec5E} for the CMFE formula $\mathcal{Q}_{\kappa,n}^{s,\gamma}[f,g]$, which approximates the integral $\mathcal{I}_\kappa[f,g]$. In this example, we consider the functions $f(x):=x^{-1/2}$ and $g(x):=x^2$ for $x\in I$, the same as those in Example \ref{egP:3}.

\begin{center}
\makeatletter
\def\@captype{figure}
\makeatother
\includegraphics[width=5.0cm]{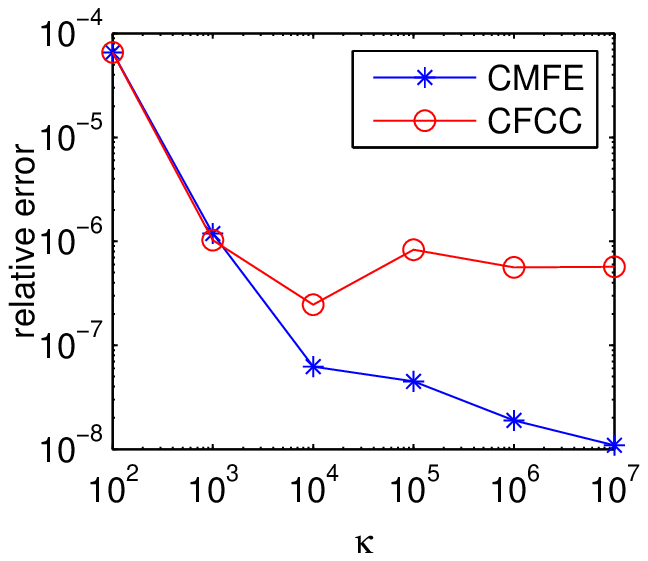}\ \hspace{2.5cm} \
\includegraphics[width=5.0cm]{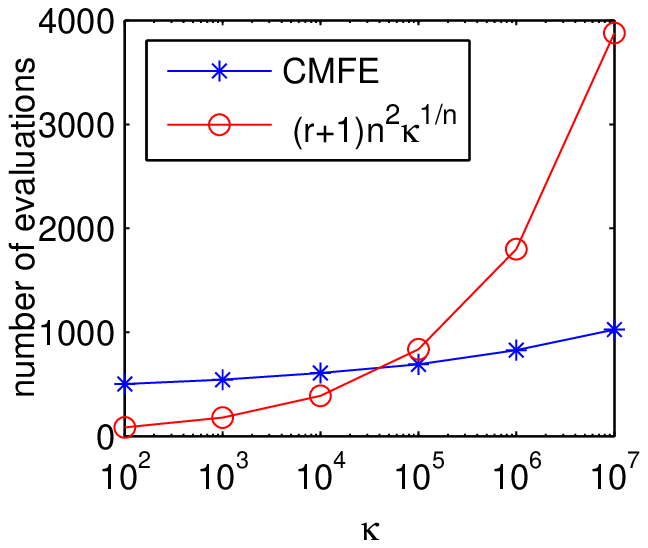}\\
\centerline{\small A:~RE of $\mathcal{Q}_{\kappa,3}^{12,0.02}[f,g]$ \hspace{4cm} B:~ $\mathcal{N}(\mathcal{Q}_{\kappa,3}^{12,0.02}[f,g])$}
\vspace{-1em}
{\small\caption{ RE and  $\mathcal{N}$ for $f(x):=x^{-1/2}$ and $g(x):=x^2$
\label{Sec6:Fig1EgE3}}}
\end{center}

We list in Table \ref{Sec6:Tab1EgE3} the RE values of the formula $\mathcal{Q}_{\kappa,n}^{12,0.02}[f,g]$ and those of the CFCC formula $\mathcal{I}^{[0,1]}_{\kappa,4,M,21}\left((f/g')\circ g^{(-1)}\right)$.
In (A) of Figure \ref{Sec6:Fig1EgE3}, we depict the RE values of the formula $\mathcal{Q}_{\kappa,3}^{12,0.02}[f,g]$, which are listed in Row 2 of Table \ref{Sec6:Tab1EgE3} and the RE values of the CFCC formula $\mathcal{I}^{[0,1]}_{\kappa,4,M,21}\left((f/g')\circ g^{(-1)}\right)$ in (A) of Figure \ref{Sec6:Fig1EgE3}, where the numbers $\mathcal{N}=505$, $545$, $609$, $693$, $829$, $1029$ for $\kappa=10^{j+1}$, $j\in\mathbb{Z}_6^+$ of functional evaluations used in the formula
are comparable to those used in the CMFE formula recorded in Row 3 of Table \ref{Sec6:Tab1EgE3}.
In Figure \ref{Sec6:Fig1EgE3} (B), we compare $\mathcal{N}(\mathcal{Q}_{\kappa,3}^{12,0.02}[f,g])$ with $(r+1)n^2\kappa^{1/n}$ with $r=1$ and $n=3$.

From Figure \ref{Sec6:Fig1EgE3} (A) and Table \ref{Sec6:Tab1EgE3}, we conclude that the accuracy of $\mathcal{Q}_{\kappa,n}^{12,0.02}[f,g]$ is higher than that of the CFCC formula.
\end{example}

From Examples \ref{egE:2} and \ref{egE:3}, we observe that the CMFE formula proposed in Section 5 has higher order of approximation accuracy than the CFCC formula when the integrand has a singularity with the index $\mu<0$ and a stationary point of the order $r>0$. Moreover, the CMFE formulas do not have to compute $g^{-1}$ for the nonlinear oscillator $g$ (which is required by the FCC and CFCC formulas) and as a result, they can save an enormous amount of computing time in comparing with the FCC and CFCC formulas.

In the next example, we shall compare the CPU time spent using the CMFE formulas, the FCC and CFCC formulas when computing oscillator integrals with a nonlinear oscillator. The CPU time is monitored by using {\it tic} and {\it toc} of Matlab. The computation of the change of variables $g(x)\to x$ is carried out using {\it fzero} of Matlab.

\begin{example}\label{egE:4}\rm
This example compare the CPU time spent when computing the integral $\mathcal{I}_\kappa[f,g]$ by using the CMFE formulas, the FCC  and CFCC formulas. We consider the functions $f_1(x):=1$, $f_2(x):=\ln{x}$ and
$g(x):=\left(\sin{(\pi x/2)}+2x\right)/3$ for $x\in I$.

\begin{center}
\footnotesize
\makeatletter
\def\@captype{table}
\makeatother
{\caption{ Comparison of CPU time  for  $f_1(x):=1$ and $g(x):=\left(\sin{(\pi x/2)}+2x\right)/3$  }
\vspace{0.2cm}
\label{Sec6:Tab1EgE4}}{
\begin{tabular}{c|c|c|c|c|c|c}
\hline
\multirow{2}{*}{$\kappa$}
&\multicolumn{3}{c|}{CMFE}&\multicolumn{3}{c}{FCC}\\
\cline{2-7}
&Time (sec.)&Approximation&$\mathcal{N}$&Time (sec.)&Approximation&$\mathcal{N}$\\
\hline
$10^2$&3.34e-2&-7.35e-3-(4.66e-3)i&12&2.84e-1&-7.35e-3-(4.66e-3)i&9\\
$10^3$&3.35e-2&1.24e-3-(1.12e-6)i&12&2.97e-1& 1.24e-3-(1.12e-6)i&9\\
$10^4$&3.31e-2&-4.59e-5+(2.27e-4)i&12&2.85e-1&-4.59e-5+(2.27e-4)i&9\\
$10^5$&3.33e-2&5.36e-7+(2.34e-5)i&12&2.83e-1& 5.36e-7+(2.34e-5)i&9\\
$10^6$&3.31e-2&-5.25e-7-(5.65e-7)i&12&2.84e-1&-5.25e-7-(5.65e-7)i&9\\
$10^7$&3.35e-2&6.31e-8+(2.20e-7)i&12&2.85e-1& 6.31e-8+(2.20e-7)i&9\\
\hline
\end{tabular}}
\end{center}

\begin{center}
\footnotesize
\makeatletter
\def\@captype{table}
\makeatother
{\caption{ Comparison of CPU time for $f_2(x):=\ln{x}$ and $g(x):=\left(\sin{(\pi x/2)}+2x\right)/3$}
\vspace{0.2cm}
\label{Sec6:Tab2EgE4}}{
\begin{tabular}{c|c|c|c|c|c|c}
\hline
\multirow{2}{*}{$\kappa$}
&\multicolumn{3}{c|}{CMFE}&\multicolumn{3}{c}{CFCC}\\
\cline{2-7}
 &Time (sec.)&Approximation&$\mathcal{N}$&Time (sec.)&Approximation&$\mathcal{N}$\\
\hline
$10^2$&4.09e-2&-1.30e-2-(4.51e-2)i&88&3.90e-1&-1.30e-2-(4.51e-2)i&73\\
$10^3$&4.10e-2&-1.31e-3-(6.43e-3)i&88&3.95e-1&-1.32e-3-(6.43e-3)i&73\\
$10^4$&4.10e-2&-1.32e-4-(8.37e-4)i&88&3.96e-1&-1.31e-4-(8.37e-4)i&73\\
$10^5$&4.10e-2&-1.32e-5-(1.03e-4)i&88&3.63e-1&-1.33e-5-(1.03e-4)i&73\\
$10^6$&4.15e-2&-1.32e-6-(1.22e-5)i&88&3.64e-1&-1.27e-6-(1.22e-5)i&73\\
$10^7$&4.16e-2&-1.32e-7-(1.42e-6)i&88&3.94e-1&-1.34e-7-(1.42e-6)i&73\\
\hline
\end{tabular}}
\end{center}

Tables \ref{Sec6:Tab1EgE4} lists
approximation values produced by the CMFE formula $\mathcal{Q}_{\kappa,3}[f_1,g]$  and the FCC formula $\mathcal{I}_{\kappa,8}^{[0,1]}\left((f_1/g')\circ g^{(-1)}\right)$ and computing time they consume.
While Table \ref{Sec6:Tab2EgE4} lists approximation values produced by the CMFE formula $\mathcal{Q}_{\kappa,6}^{6,0.02}[f_2,g]$ and the CFCC formula $\mathcal{I}^{[0,1]}_{\kappa,8,10,12}\left((f_2/g')\circ g^{(-1)}\right)$ and computing time they consume.
Both tables show that the CMFE formula consumes significantly less CPU time than the FCC, CFCC formulas when they produce comparable approximation results even when the CMFE formula uses more functional evaluations.

\end{example}

\section{Concluding remarks}

We develop in this paper composite quadrature formulas for computing highly oscillatory integrals defined on a finite interval with both singularities and stationary points. The partitions of the integration interval used in the composite quadrature formulas are designed according to the degree of oscillation and the singularity. In each of the subintervals, we use piecewise polynomial interpolants to approximate the integrand to form two classes of formulas having polynomial (resp. exponential) order of convergence by using fixed (resp. variable) number of interpolation nodes. Numerical experiments are carried out to confirm the theoretical results on the accuracy of the proposed formulas and to compare them with existing methods. Numerical results show that the proposed formulas outperform the existing methods in both approximation accuracy and computational efficiency.




\begin{thebibliography}{99}

\setlength{\itemsep}{- 1mm}

\bibitem{Abramowitz}M. Abramowitz amd I. A. Stegun, \emph{Handbook of Mathematical Functions},
Dover, New York, 1965.

\bibitem{Asheim}A. Asheim and D. Huybrechs, Asymptotic analysis of numerical steepest descent with path approximations, Found. Comput. Math., {\bf 10} (2010), 647--671.

\bibitem{Berend}D. Berend and T. Tassa, Improved bounds
on Bell numbers and on moments of sums of random variables, Prob. Math. Statist., {\bf 30} (2010),
185--205.

\bibitem{Xu2} Z. Chen, B. Wu and Y. Xu, Error control strategies for  numerical integrations in
fast collocation methods, Northeast. Math. J., {\bf 21} (2005), 233--252.


\bibitem{Huybrechs} A. Dea\~{n}o and D. Huybrechs, Complex Gaussian quadrature of oscillatory integrals,
Numer. Math., {\bf 112} (2009), 197--219.



\bibitem{Graham1}V. Dom\'{i}nguez, I. G. Graham and V. P. Smyshlyaev, Stability and error estimates for
Filon-Clenshaw-Curtis rules for highly-oscillatory integrals, IMA J. Numer. Anal.,
{\bf 31} (2011), 1253--1280.

\bibitem{Graham}V. Dom\'{i}nguez, I. G. Graham and T. Kim, Filon-Clenshaw-Curtis rules
for highly-oscillatory integrals with algebraic singularities and stationary points, SIAM J. Numer. Anal., {\bf 51} (2013), 1542--1566.
\bibitem{Dominguez}V. Dom\'{i}nguez, Public domain code,  http://www.unavarra.es/personal/victor\_dominguez/clenshawcu-rtisrule.\linebreak[0]

\bibitem{Filon} L. Filon, On a quadrature formula for trigonometric integrals, Proc. Roy. Soc. Edinburgh, {\bf 49} (1928), 38--47.
\bibitem{Flinn}E. A. Flinn, A modification of Filon's method of numerical integration,
J. ACM, {\bf 7} (1960), 181--184.

\bibitem{Gradshteyn}I. S. Gradshteyn and I. M. Ryzhik, Table of Integrals, Series, and Products,
Academic Press, New York, 1994.


\bibitem{Huybrechs1} D. Huybrechs and S. Vandewalle, On the evaluation of highly oscillatory integrals
by analysis continuation, SIAM J. Numer. Anal., {\bf 44} (2006),  1026--1048.

\bibitem{Iserles2}A. Iserles  and S. P. N{\o}rsett, On quadrature methods for highly oscillatory
integrals and their implementation. BIT Numer. Math.,  {\bf 44} (2004), 755--772.

\bibitem{Iserles3}A. Iserles, On the numerical quadrature of highly oscillating integrals
I: Fourier transforms, IMA J. Numer. Anal., {\bf 24} (2004), 365--391.

\bibitem{Iserles4}A. Iserles, On the numerical quadrature of highly oscillating integrals
II: Irregular oscillators, IMA J. Numer. Anal., {\bf 25} (2005), 25--44.
\bibitem{Iserles}A. Iserles  and S. P. N{\o}rsett, Efficient quadrature of highly-oscillatory
integrals using derivatives, Proc. R. Soc. A, {\bf 461} (2005), 1383--1399.

\bibitem{Iserles1}A. Iserles  and S. P. N{\o}rsett, Quadrature methods for multivariate highly
oscillatory integrals using derivatives, Math. Comput., {\bf 75} (2006),  1233--1258.

\bibitem{Johnson} W. P. Johnson, The curious history of Fa\`{a} di Bruno's formula, Amer. Math. Month. {\bf 109} (2002),
217--234.

\bibitem{Xu1}H. Kaneko and Y. Xu, Gauss-type quadratures for weakly singular integrals and their
application to Fredholm integral equations of the second kind, {Math. Comput.}, {\bf 62} (1994), 739--753.

\bibitem{Luke}Y. L. Luke, On the computation of oscillatory integrals,  Proc. Camb. Philos. Soc., {\bf 50} (1954), 269--277.

\bibitem{Levin}D. Levin, Procedures for Computing One- and Two-Dimensional Integrals of Functions
with Rapid Irregular Oscillations, Math. Comput. {\bf 38} (1982), 531--538.

\bibitem{Levin1}D. Levin, Analysis of a collocation method for integrating rapidly oscillatory functions, J. Comput. Appl. Math., {\bf 78} (1997), 131--138.

\bibitem{Melenk} J. M. Melenk, On the convergence of Filon quadrature, J. Comput. Appl. Math, {\bf 234} (2010),
1692--1701.

\bibitem{Olver} F. W. J. Olver, \emph{Asymptotics and Special Functions}, Macmillan, New York, 1974.

\bibitem{Olver2}S. Olver, On the quadrature of multivariate highly oscillatory integrals over nonpolytope
domains, Numer. Math., {\bf 103} (2006), 643--665.

\bibitem{Olver3}S. Olver, Moment-free numerical approximation of highly oscillatory functions,
IMA. J. Numer. Anal., {\bf 26} (2006), 213--227.

\bibitem{Olver1}S. Olver, Moment-free numerical approximation of highly oscillatory integrals
with stationary points. Euro. J.  Appl. Math., {\bf 18} (2007), 435--447.

\bibitem{Olver5}S. Olver, GMRES for the differentiation operator,
SIAM J. Numer. Anal., {\bf 47} (2009), 3359--3373.

\bibitem{Olver4}S. Olver, Fast, numerically stable computation of oscillatory integrals with stationary points.
BIT Numer. Math., {\bf 50} (2010), 149--171.

\bibitem{Riordan} J.~Riordan, An introduction to combinatorial analysis, Wiley, New York, 1958, 35--37.

\bibitem{Rennie} B. C. Rennie and  A.J. Dobson, On Stirling Numbers of the Second Kind,  J. Comb. Theory,
 {\bf 7}(1969), 116--121.

\bibitem{Rice} J.~R.~Rice, \emph{On the degree of convergence of nonlinear spline approximation}, Approx. with Emphasis on Spline Functions (I. J. Schoenberg, ed.), Academic Press, New York, 1969, pp. 349--565.

\bibitem{Roman} S.~Roman, The formula of Fa\`{a} di Bruno, {Am. Math. Mon.}, {\bf 87} (1980), 805--809.

\bibitem{Stein}E. Stein, \emph{Harmonic Analysis$:$ Real-Variable Method$,$ Orthogonality$,$
and Oscillatory Integral}, Princeton, New Jersey, 1993.

\bibitem{Schwab}C. Schwab, Variable order composite quadratures for singular and nearly singular integrals,
Computing, {\bf 53} (1994), 173--194.

\bibitem{Wilf}H. S. Wilf, Generatingfunctionology, Academic Press(2nd, edi), Boston, 1994.

\bibitem{Xiang}S. Xiang, Efficient Filon-type methods for
$\int_a^bf(x)\exp{\set{{\rm i}\omega g(x)}}{\rm d}x$, Numer. Math., {\bf 105} (2007), 633--658.

\bibitem{Xiang1}S. Xiang, On the Filon and Levin methods for
highly oscillatory integral $\int_a^bf(x)\exp{\set{{\rm i}\omega g(x)}}{\rm d}x$,
J. Comput. Appl. Math., {\bf 208} (2007), 434--439.

\end{thebibliography}
\end{document}